\documentclass[11pt]{article}

\usepackage{hyperref}
\usepackage{graphicx}
\usepackage{mathrsfs}
\hypersetup{colorlinks, linkcolor=blue, urlcolor=red, filecolor=green, citecolor=blue}

\usepackage{amsmath}
\usepackage{amssymb}
\usepackage{latexsym}
\usepackage{amsthm}
\usepackage{fullpage}
\usepackage{graphicx}
\usepackage{color}

\usepackage[utf8]{inputenc}

\graphicspath{{figures/}}

 {\begin{list}{}%
         {\setlength{\leftmargin}{#1}}%
         \item[]%
 }
 {\end{list}}

\newtheorem{theorem}{Theorem}[section]
\newtheorem{proposition}[theorem]{Proposition}

\newtheorem{definition}{Definition}[section]
\newtheorem{remark}[theorem]{Remark}
\newtheorem{assumption}{Assumption}[section]

\author{Laura Lauerbach\footnote{Institute of Mathematics,
University of W\"{u}rzburg, Emil-Fischer-Str.~40, 97074 W\"{u}rzburg, Germany and Institute of Mathematics,
University of Kassel, Heinrich-Plett-Str.~40, 34132 Kassel, Germany},
Anja Schl\"{o}merkemper\footnote{Institute of Mathematics,
University of W\"{u}rzburg, Emil-Fischer-Str.~40, 97074 W\"{u}rzburg, Germany}
}

\title{Derivation of a variational model for brittle fracture from a random heterogeneous particle chain}%Derivation of a Griffith' type fracture energy from a random heterogeneous particle chain}

\begin{document}
\maketitle

\begin{abstract}
A mathematical continuum limit of the interaction energy of a random particle chain is shown to yield new insight into the effect of microscopic heterogeneities on macroscopic fracture laws in brittle materials. We derive a formula which yields that either elastic behaviour or a crack is energetically preferred. The formula explicitly shows the dependence on the boundary condition and the microstructure of the chain. The mathematical analysis is based on a variational convergence ($\Gamma$-convergence) of convex-concave potentials together with ergodic theorems which are common tools in stochastic homogenization.
\end{abstract}
\medskip


%\tableofcontents

\section{Introduction}

Fracture in brittle materials is often modeled in a variational framework, which allows to compare elastic energy contributions and contributions due to the creation of surface. As in Griffith' theory of fracture, onset of fracture is predicted by a formula involving elastic material constants and surface energy contributions. As was pointed out by Griffith \cite{Griffith2021}, microscopic flaws have an important impact on the yield stress that causes brittle fracture. 

In this article we investigate the effect of microscopic heterogeneities on the onset of fracture in the setting of a one-dimensional toy model, which consists of particles that interact through some convex-concave potential with constant energy at large distances between the particles. We strive for considering a large class of interaction potentials, which includes potentials of convex-concave shape with constant energy in the limit of large particle distance and singular behaviour at zero distance. Typical examples include Lennard-Jones potentials. Further, we allow for defects, weak interaction potentials or composite materials. The aim is to understand the effect of such heterogeneities on the effective behaviour as the number of particles tends to infinity. The passage from the discrete/microscopic to a continuous/macroscopic system is performed in the context of $\Gamma$-convergence, a notion of variational convergence which is suitable for minimization problems of energy functionals depending on a parameter like the number of particles. For an introduction to $\Gamma$-convergence and related literature we refer to \cite{Lauerbach-Diss}, which is the first author's PhD thesis upon which this article is based. In fact, large parts of this article are identical to \cite[Chapter~5]{Lauerbach-Diss}.

The results presented in this work for random heterogeneous particle chains extend earlier results in a periodic setting \cite{LauerbachSchaeffnerSchloemerkemper2017}, to which we also refer for an overview of related results in the homogeneous setting, cf.\ also the end of Section~\ref{Sec:GammaLimitRescaled}.

In the next section we will introduce the random heterogeneous particle chain and the class of interaction potentials in detail. Further we introduce the variational model, i.e.\ the energy functional that summarizes all the interaction potentials between the nearest neighbours. Inspired by earlier work \cite{BraidesLewOrtiz2006} 
we consider a rescaled version of the energy which ensures the same scaling of surface and bulk contributions to the energy and thus allows to obtain a comparison of fracture and elastic bulk energy in the continuum limit. 

In Section~\ref{sec:3} we state our main result (Theorem~\ref{Thm:rescaled}) and an accompanying compactness result; the proofs are provided in Section~\ref{sec:proofs}. The obtained $\Gamma$-convergence result yields that the limiting energy is finite whenever the displacement satisfies certain regularity and boundary conditions and has at most finitely many jumps which are increasing. The limiting energy has two terms: the first term represents a linear elastic energy with an elastic modulus that is obtained by the inverse of the expectation of the inverse of the elastic modulus of the heterogeneous interaction potentials at the minimizers; the second term reflects the energy needed to create a crack, i.e.\ a jump of the displacement vector. The second term, which can be interpreted as a (zero-dimensional) surface energy, is shown to be determined by the weakest interaction potential of the random heterogeneous particle chain. This mathematically proved result coincides with the physical intuition that the weakest bond determines the onset of fracture.

\section{Discrete model -- stochastic Lennard-Jones type interactions} %\label{sec:2}

\begin{figure}[t]
\centering
		\includegraphics[width=0.6\linewidth]{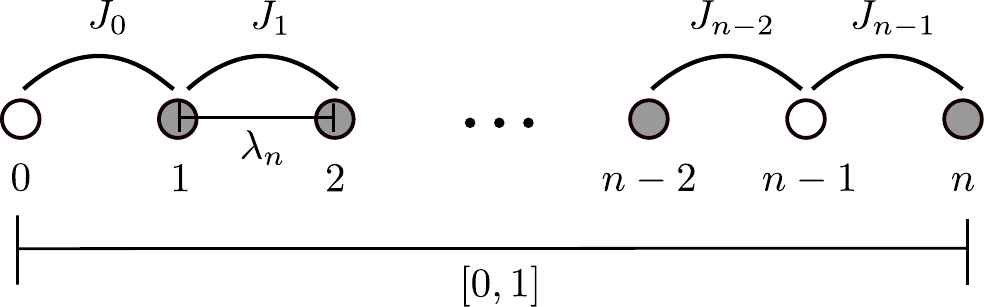}
		\caption{Chain of $n+1$ randomly distributed particles with reference position $x_n^{i}=i\lambda_n$. The potentials $J_i$ describe the nearest neighbour interaction of particle $i$ and $i+1$. The characteristic length scale is $\lambda_n=\frac{1}{n}$ and the interval is $[0,1]$.}
\label{fig:kette}
\end{figure}

Let $\lambda_n\mathbb{Z}\cap[0,1]$ be a one dimensional lattice, where $\lambda_n=\frac{1}{n}$. We regard this as a chain of $n+1$ particles, see Figure~\ref{fig:kette}. The reference position of the $i$-th particle is referred to as $x_n^i:=i\lambda_n$ and the deformation of the particles is denoted by $u_n: \lambda_n\mathbb{Z}\cap[0,1]\rightarrow\mathbb{R}$, where we write $u(x_n^i)=u^i$ for short. In the passage from discrete systems to their continuous counterparts  piecewise affine interpolations of the deformation $u$ are used. We define 
\begin{align*}
	\mathcal{A}_n:=\left\lbrace u \in C([0,1]): u\ \text{is affine on}\ (i,i+1)\lambda_n,\ i\in\{0,1,\ldots,n-1\} \right\rbrace
\end{align*}
as the set of all piecewise affine functions which are continuous. The interaction potentials between the particles of this chain can be quite general in our analysis. Moreover these potentials can be random due to e.g.\ a heterogeneous distribution of different particles, see below.

%\subsection{Lennard-Jones type potentials} \label{sec:LJ}

We consider interaction potentials belonging to a large class $\mathcal{J}(\alpha,b,c,d,\Psi,\eta)$ of functions that includes in particular the classical Lennard-Jones potential, which is the reason why we refer to the considered interaction potentials as being of \textit{Lennard-Jones type}, cf.\ also \cite{Lauerbach-Diss,unserPaper1}. It is defined as follows.
\begin{definition}
 Fix $\alpha\in(0,1]$, $b>0$, $c>0$, $d\in(1,+\infty)$, $\eta>0$ and a convex function $\Psi:\mathbb R\to[0,+\infty]$ satisfying
 \begin{equation}
 \label{domrandpsi}
  \lim_{z\to0^+} \Psi(z)=+\infty.
 \end{equation} 
We denote by $\mathcal J=\mathcal J(\alpha,b,c,d,\Psi,\eta)$ the class of functions $J:\mathbb R\to \mathbb R\cup\{+\infty\}$ which satisfy the following properties:
\begin{itemize}
	\item[(LJ1)] (Regularity and asymptotic decay) It is $J\in C^3$ on its domain and
	\begin{align*}
	 \lim_{z\to0^+} J(z)=\infty\quad \mbox{as well as}\quad J(z)=\infty\mbox{ for $z\leq0$.}
	\end{align*}
	\item[(LJ2)] (Convex bound, minimum and minimizer) $J$ has a unique minimizer $\delta$ with $\delta\in (\tfrac1d,d)$ and $J(\delta)<0$, and is strictly convex on $(0,\delta)$. Moreover, $\|J\|_{L^\infty(\delta,\infty)}<b$ and it holds	\begin{align*}
	%\label{LJ2abschaetzung}
		\tfrac1d\Psi(z)-d\leq J(z)\leq d\max\{\Psi(z),|z|\}\quad\text{for all}\ z\in(0,+\infty).
	\end{align*}
	\item[(LJ3)] (Asymptotic behaviour) It holds
	\begin{align*}
	 \lim_{z\to\infty} J(z)=0.
	\end{align*}
	\item[(LJ4)] (Harmonic approximation near ground state) For $|z-\delta|<\eta$, it holds true that
	\begin{align*}
	J(z)-J(\delta)\geq \frac{1}{c}(z-\delta)^2.
	\end{align*}
\end{itemize}
\end{definition}

\begin{remark}
	\label{Rm:alphaneu}
	%This remark gives some comments on the new properties (LJ1) and (LJ4).
	\begin{itemize}
		\item[(i)] The regularity condition in (LJ1) is not sharp. In principle, it would suffice to require $J\in C^2$. However, the $C^3$ regularity allows to use the Lagrange form of the remainder in a Taylor-expansion in the proof of Theorem \ref{Thm:rescaled}, which is more convenient, cf.\ also \eqref{defCkappa}. Moreover, hypothesis (R1) (see below) would be more difficult to formulate if we just assumed $J\in C^2$. 
		\item[(ii)] A harmonic approximation, like in (LJ4), of a function $f\in C^2$ at the minimum point $x_0$ is always possible as long as it holds true that $f''(x_0)>0$. 
		
		\item[(iii)] With the definition %\label{pr:alpha}
		\begin{align*}
		\alpha:=\left.\dfrac12\dfrac{\partial^2J(z)}{\partial z^2}\right|_{z=\delta},
		\end{align*}
		it follows from (LJ4) and (ii) that $\alpha>C_\alpha$ for a constant $C_{\alpha}>0$ uniformly for all potentials in the class $\mathcal{J}(\alpha,b,c,d,\Psi,\eta)$. 
		%This follows directly from (ii), where it was shown that $c$ has to be chosen in such a way that it holds true that $\tilde{C}=\frac12 f''(0)- \frac{1}{c}>0$.
		
		\item[(iv)] The assumption in (LJ4) contains a uniform bound to handle the situation of infinitely many potentials and is needed for the stochastic setting. For finitely many different potentials, (LJ4) is fulfilled automatically.
	\end{itemize}
\end{remark}
 
For further remarks on the above assumptions and for examples of Lennard-Jones type potentials we refer to \cite{unserPaper1}. 
\medskip 

The randomness in our model enters through random nearest neighbour interactions. The random interaction potentials $\{J(\omega,i,\cdot)\}_{i\in\mathbb{Z}}$, $J(\omega,i,\cdot):\mathbb{R}\to(-\infty,+\infty]$, are of Lennard-Jones type, as defined above; they are assumed statistically homogeneous and ergodic, which is a standard assumption in the context of stochastic homogenization theory. To this end, let $(\Omega,\mathcal{F},\mathbb{P})$ be a probability space, which can be discrete or continuous with uncountably many different elements in the set $\Omega$. We assume that the family $(\tau_i)_{i\in\mathbb{Z}}$ of measurable mappings $\tau_i:\Omega\rightarrow\Omega$ is an additive group action, i.e., $\tau_0\omega=\omega$ for all $\omega\in\Omega$ and $\tau_{i_1+i_2}=\tau_{i_1}\tau_{i_2}$ for all $i_1,i_2\in\mathbb{Z}$. 
Additionally, we require that the group action is measure preserving, that is $\mathbb{P}(\tau_iB)=\mathbb{P}(B)$ for every $B\in\mathcal{F}$, $i\in\mathbb{Z}$ (stationarity). Moreover, we assume ergodicity, i.e., for all $B\in\mathcal{F}$, it holds $(\tau_i(B)=B\ \forall i\in\mathbb{Z})\Rightarrow\mathbb{P}(B)=0\ \text{or}\ \mathbb{P}(B)=1$.

\medskip

% \begin{figure}[t]
% \centering
% 		\includegraphics[width=0.7\linewidth]{./figures/Random.pdf}
% 		\caption{Randomly arranged chain of particles. The nearest neighbour interaction potential of two grey particles is labelled by $J_a$, that of two white particles by $J_b$ and that between a white and a grey one by $J_b$. Since the particles are randomly distributed, this holds for the potentials as well.}
% \label{fig:random}
%\end{figure}
We define $\tilde{J}:\Omega\to\mathcal{J}(\alpha,b,c,d,\Psi,\eta),\ \omega\mapsto \tilde{J}(\omega)(\cdot)=:\tilde{J}(\omega,\cdot)$, measurable in $\omega$. %This maps the sample space into the set of Lennard-Jones potentials. 
Then, we set
\begin{align*}
	J(\omega,i,\cdot) :=\tilde{J}(\tau_{i}\omega,\cdot)\quad\text{ for all } i\in\mathbb{Z},\ \omega\in\Omega
\end{align*}
as the random interaction potential between particles $i$ and $i+1$.
Hence every mapping $\tau_i:\Omega\to\Omega$ of the group action is assigned to a particle of the chain and is used to relate the different particles to different elements of the sample space and therefore to different interaction potentials. 
In the following, we denote $\tilde{J}$ simply by $J$, for better readability. We also introduce some notation for the minimizers
\begin{align*}
\delta(\omega):=\mathrm{argmin}_{z\in\mathbb{R}}\left\{\tilde{J}(\omega,z) \right\}, \quad \delta(\tau_i\omega):=\mathrm{argmin}_{z\in\mathbb{R}}\left\{J(\omega,i,z) \right\}, \quad \text{for all}\ i\in\mathbb{Z}.
\end{align*}
Further, we set for all $\omega\in\Omega$
\begin{align} \label{alphaomega}
\alpha(\omega)&:=\left.\dfrac12\dfrac{\partial^2J(\omega,z)}{\partial z^2}\right|_{z=\delta(\omega)},
\end{align}
and for $0<\kappa<\frac{1}{d}$ %\label{pr:ckappa}
\begin{align} \label{defCkappa}
C^{\kappa}(\omega)&:=\sup\left\{\left|\frac{\partial^3 J}{\partial z^3}(\omega,z)\right|\, :\, z\in\left[\delta(\omega)-\kappa,\delta(\omega)+\kappa\right] \right\}.
\end{align}
Since $\frac{\partial^3 J}{\partial z^3}(\omega,\cdot)$ is continuous due to (LJ1), $C^{\kappa}(\omega)<\infty$ holds true for every $\omega\in\Omega$ and every $0<\kappa<\frac{1}{d}$.

\begin{remark}
	\label{Rm:integrabilityandexpectationvalues2}
	
	Due to Remark~\ref{Rm:alphaneu} (iii), it holds true that $0<(\alpha(\omega))^{-1}<C$ and this implies integrability of the random variable $(\alpha(\omega))^{-1}$ and therefore the expectation value of $(\alpha(\omega))^{-1}$ exists. It is denote by $\mathbb{E}[\alpha^{-1}]$. 
\end{remark}

The assumptions on the stochastic setting of the chain with Lennard-Jones type interaction potentials are summarized in

\begin{assumption}
	\label{Ass:stochasticLJ}
	Fix $\alpha\in(0,1]$, $b>0$, $c>0$, $d\in(1,\infty)$, $\eta>0$ and  a convex function $\Psi:\mathbb R\to[0,\infty]$ satisfying \eqref{domrandpsi}. Let $(\Omega,\mathcal{F},\mathbb{P})$ be a probability space and $(\tau_i)_{i\in\mathbb{Z}}$ be a family of stationary and ergodic group actions. The random variable $J:\Omega\to\mathcal J(\alpha,b,c,d,\Psi,\eta)$ is measurable and satisfies
\begin{itemize}
	\item[(R1)] (Third derivative near ground state) There exists $0<\kappa^*<\frac{1}{d}$ such that $\mathbb{E}[C^{\kappa^*}]<\infty$. As a direct consequence, it also holds true that $\mathbb{E}[C^{\kappa}]<\infty$ for every $\kappa<\kappa^*$, by definition of $C^{\kappa}$, see \eqref{defCkappa}.
	\item[(R2)] (Uniform convergence of the asymptotic decay) It holds true that
	\begin{align*}
	\lim\limits_{z\to\infty}\max_{\omega\in\Omega}\left|J(\omega,z) \right|=0.
	\end{align*}
\end{itemize}
\end{assumption}

We remark that the assumptions (R1) and (R2) are automatically satisfied when dealing with finitely many different potentials. \medskip

Next we introduce the energy functional considered. For a given piecewise affine deformation $u\in\mathcal{A}_n$ the canonical energy of the chain of particles with nearest neighbour interactions reads
\begin{align*}%\label{eq:energy}
	E_n(\omega,u):=\sum_{i=0}^{n-1}\lambda_nJ\left(\omega,i,\dfrac{u^{i+1}-u^{i}}{\lambda_n}\right)
\end{align*}
accompanied by the boundary conditions $u(0)=0$, $u(1)=\ell$ for some given $\ell>0$.
\medskip

As pointed out earlier, we consider here a rescaled version of the energy which ensures that bulk and surface contributions scale in the same way in order to overcome a separation of scales, cf.\ \cite{BraidesLewOrtiz2006,BraidesTruskinovsky2008,ScardiaSchloemerkemperZanini2012} for the deterministic setting. 
In the random setting we transform the deformation $u$ into a properly rescaled displacement $v$ as follows:
\begin{align*}
%\label{rescaledvariable}
v^{i}:=\dfrac{u^{i}-\sum_{k=0}^{i-1}\lambda_n\delta(\tau_k\omega)}{\sqrt{\lambda_n}}\quad\text{for all}\ i\in\{0,\ldots,n\}.
\end{align*}
Hence, we have in particular that $v$ is piecewise affine and
\begin{align*}
\dfrac{u^{i+1}-u^{i}}{\lambda_n}=\dfrac{v^{i+1}-v^{i}}{\sqrt{\lambda_n}}+\delta(\tau_i\omega).
\end{align*}
The rescaled variable shifts the minimizer of the potential $J(\omega,\cdot)$ to the position $v'=0$. Further, a constant term is added to the potential, which results in the final rescaled energy 
\begin{align*}
%\label{rescaledenergy}
E_n(\omega,v):=\sum_{i=0}^{n-1}\left(J\left(\tau_i\omega,\dfrac{v^{i+1}-v^{i}}{\sqrt{\lambda_n}}+\delta(\tau_i\omega)\right)-J\left(\tau_i\omega,\delta(\tau_i\omega)\right) \right).
\end{align*}
 
Next we adapt the Dirichlet boundary conditions $u(0)=0$ and $u(1)=\ell$.  Following the ideas of \cite{ScardiaSchloemerkemperZanini2012}, adjusted to our stochastic setting, we focus on some sequence $(\ell_n)\subset\mathbb{R}$ with $\ell_n\rightarrow\mathbb{E}[\delta]$, satisfying $\ell_n>\mathbb{E}[\delta]$ for every $n\in\mathbb{N}$ and %\label{pr:gamman}
\begin{align}
\label{bedingunggamma}
\gamma_n:=\dfrac{\ell_n-\sum_{k=0}^{n-1}\lambda_n\delta(\tau_k\omega)}{\sqrt{\lambda_n}}\rightarrow\gamma,
\end{align}
for some $\gamma\in\mathbb{R}$. The new Dirichlet boundary conditions then read $v(0)=0$ and $v(1)=\gamma_n$. For definiteness, we assume that
\begin{align*}
\ell_n&>\frac{1}{n}\sum_{i=0}^{n-1}\delta(\tau_i\omega)\quad \text{for every}\ n\in\mathbb{N}.
\end{align*}
Hence we have that $\gamma\geq0$ and $\gamma_n>0$ for all $n\in\mathbb{N}$. Taking the boundary conditions into account, we obtain 
 the rescaled energy functional $E_n^{\gamma_n}:\Omega\times L^1(0,1)\rightarrow(-\infty,+\infty]$ with
\begin{align*}
E_n^{\gamma_n}(\omega,v):=\begin{cases}
E_n(\omega,v)&\quad\text{for}\ v\in\mathcal{A}_n^{\gamma_n}(0,1), v(0)=0,v(1)=\gamma_n,\\
+\infty&\quad\text{else}.
\end{cases}
\end{align*}
Next we assert the $\Gamma$-limit result of this energy as the number of particles tends to $\infty$.

\section{Continuum limit of the rescaled energy -- the main result} \label{sec:3}

 We briefly recall notation related to the space $SBV$. For given $\gamma>0$, we denote by $SBV^\gamma(0,1)$ the space of special functions of bounded variations in $(0,1)$ with additional boundary constraint $v(0-)=0$ and $v(1+)=\gamma$. For $v\in SBV^\gamma(0,1)$, we set $S_v:=\{x\in[0,1]\,|\, [v](x)\neq0\}$ with $[v](x):=v(x+)-v(x-)$ where $v(x+)$ for $x\in[0,1)$ and $v(x-)$ for $x\in(0,1]$ are the right and left essential limits at $x$.

At first, we state a compactness result for functions with equibounded energy, which ensures the convergence of minimizers in the sense of the main theorem of $\Gamma$-convergence. We provide its proof in Section~\ref{sec:4.1}.
\begin{theorem}
	\label{Thm:compactnessgammarescaled}
	Let Assumption \ref{Ass:stochasticLJ} be satisfied. Let $\gamma_n$ be such that \eqref{bedingunggamma} holds true and let $(v_n)$ be a sequence of functions such that
%	\begin{align*}
	$\sup_nE_n^{\gamma_n}(\omega,v_n)<+\infty$
	%\end{align*}
	for every $\omega\in\Omega$. Then, there exist a subsequence $(v_{n_k})$ and $v\in SBV^{\gamma}(0,1)$ such that $v_{n_k}\rightarrow v$ in $L^1(0,1)$  and
	\begin{align*}
	v'\in L^2(0,1),\quad \#S_v<+\infty,\quad[v]\geq 0\ in\ [0,1].
	\end{align*}
	Moreover, there exists a finite set $S\subset[0,1]$ such that $v_{n_k}\rightharpoonup v$ locally weakly in $H^1((0,1)\setminus S)$.
\end{theorem}

For convenience, the properties shown in the compactness result are collected in the definition %\label{pr:sbvcgamma}
\begin{align*}
SBV_c^{\gamma}(0,1):=\left\{v\in SBV^{\gamma}(0,1)\, :\,v'\in L^2(0,1),\ \#S_v<+\infty,\ [v]\geq0\ \text{in}\ [0,1] \right\}.
\end{align*}
We are now in a position to state the main result of this article; it is proved in Section~\ref{Sec:GammaLimitRescaled}.

\begin{theorem}
	\label{Thm:rescaled}
	Let Assumption \ref{Ass:stochasticLJ} be satisfied and let $\gamma_n$ be such that \eqref{bedingunggamma} holds true. Then, there exists an $\Omega'\subset\Omega$ with $\mathbb{P}(\Omega')=1$ such that for all $\omega\in\Omega'$ the sequence $(E_n^{\gamma_n})$ $\Gamma$-converges with respect to the $L^1(0,1)$-topology to the functional $E^{\gamma}$ given by
	\begin{align*}
	E^{\gamma}(v):=\begin{cases}
	\underline{\alpha}\displaystyle \int_{0}^{1}\left|v'(x) \right|^2\,\mathrm{d}x+\beta\#S_v&\quad\text{if}\ v\in SBV_c^{\gamma}(0,1),\\[1mm]
	+\infty&\quad\text{otherwise,}
	\end{cases}
	\end{align*}
	where %\label{pr:alphabeta} 
	$\underline{\alpha}:=\left(\mathbb{E}\left[\alpha^{-1}\right]\right)^{-1}$ 
	with $\alpha(\omega)$ as in \eqref{alphaomega},
	and $\beta:=\inf\left\{-J(\omega,\delta(\omega))\, :\, \omega\in\Omega\right\}.$
	Moreover, for $\gamma>0$ it holds true that
	\begin{align*}
	\lim\limits_{n\rightarrow\infty}\inf_v E_n^{\gamma_n}(\omega,v)=\min_vE^{\gamma}(v)=\min\left\{\underline{\alpha}\gamma^2,\beta\right\}.
	\end{align*}
\end{theorem}
\vspace{1mm}
%\begin{remark}
%	\label{Rem:comparisonrescaled}
Hence, depending on the parameters $\underline{\alpha}$, $\beta$ and $\gamma$ the minimizer of the system either shows elastic behaviour or one crack. While $\underline{\alpha}$ is given in terms of an expectation value, the constant $\beta$ of the fracture part of the energy is not. The infimum in the definition of $\beta$ yields that the energy needed to create a crack is determined by the weakest bond of the chain. The location of a crack, however, remains unknown, which is in line with the proof of the $\Gamma$-limit result. More precisely, the construction of the recovery sequence in the limsup inequality is based on studying certain sequences of weakest bonds \eqref{hnepsilon}. Due to the ergodic theorem this construction can be performed for any position of the crack with probability one and yields the term $\beta$ in the continuum limit, cf.\ \eqref{limitbetan}.

\medskip

The periodic setting can be considered as a special case of the random setting. We remark that then the elastic constant $\underline{\alpha}$ is given as the harmonic mean of $\alpha_i:=\frac{1}{2}J_i''(\delta_i)$, i.e.\ of the quadratic term in the Taylor expansion, which is consistent with the corresponding result in \cite{LauerbachSchaeffnerSchloemerkemper2017}.	In \cite{BraidesGelli2006}, where a periodic setting with truncated parabolas is considered, the prefactor of the quadratic energy replaces the coefficient $\alpha_i$ of the Taylor series. The elastic constant $\underline{\alpha}$ thus is also its harmonic mean. 
In the periodic setting, the infimum in the formular for $\beta$ becomes a minimium, cf.~\cite{LauerbachSchaeffnerSchloemerkemper2017}. This is also in accordance with \cite{BraidesGelli2006}, where the constant of the fracture part of the limiting energy is the minimum over the truncation heights.
We conclude that our $\Gamma$-convergence result extends the earlier result to the stochastic one in a natural and consistent way. 
%\end{remark}

%It is a boundary layer energy and corresponds to the bond of the chain which is broken or missing due to the crack, see Figure~\ref{fig:beta}.

% \begin{figure}[t]
% 	\centering
% 	\includegraphics[width=0.7\textwidth]{./figures/beta.pdf}  
% 	\caption{Chain with a crack in the middle. The dotted line is the broken bond due to the crack and matches the term $\beta$ in the energy functional, cf.~Theorem~\ref{Thm:rescaled}.}
% 	\label{fig:beta}
% \end{figure}

\section{Proofs} \label{sec:proofs}

We start by defining two functions, which represent sample averages of $\alpha^{-1}$ and $C^\kappa$, and  consider their limits in the next proposition. For arbitrary $N\in\mathbb{N}$ we set
\begin{align*}
\alpha^{-1,(N)}(\omega,A)&:=\dfrac{1}{|NA\cap\mathbb{Z}|}\sum_{i\in NA\cap\mathbb{Z}}\dfrac{1}{\alpha(\tau_i\omega)},\\[2mm]
C^{\kappa,(N)}(\omega,A)&:=\dfrac{1}{|NA\cap\mathbb{Z}|}\sum_{i\in NA\cap\mathbb{Z}}C^\kappa(\tau_i\omega).
\end{align*}

\begin{proposition}
	\label{Prop:averages2}
	Let Assumption \ref{Ass:stochasticLJ} be satisfied. Then there exists an $\Omega'\subset\Omega$ with $\mathbb{P}(\Omega')=1$ such that for all $\omega\in\Omega'$, 
	all $\kappa<\kappa^*$ and for all $A=[a,b]$ with $a,b\in\mathbb{R}$ the limits
	\begin{align*}
	\mathbb{E}[\alpha^{-1}]&=\lim\limits_{N\to\infty}\alpha^{-1,(N)}(\omega,A),\\[2mm]
	\mathbb{E}[C^\kappa]&=\lim\limits_{N\to\infty}C^{\kappa,(N)}(\omega,A)
	\end{align*}
	exist in $\overline{\mathbb{R}}$ and are independent of $\omega$ and the interval $A$. 
\end{proposition}

\begin{proof}
This is a direct consequence of a variant of Birkhoff's ergodic theorem \cite[Section~6.2]{Krengel}; we refer to \cite[Proposition~5.5.]{Lauerbach-Diss} for details.
\end{proof}
% \begin{proof}
% 	The proof is fully analogous to the proof of Proposition~\ref{Prop:averages1}, but with the following adaptations.
	
% 	Integrability of the random variables is now guaranteed by Remark~\ref{Rm:integrabilityandexpectationvalues2} and (R1). 
	
% 	The proof for $\alpha^{-1,(N)}(\omega,A)$ is done with the set $\Omega_{\alpha^{-1}}$. Note that $\alpha^{-1,(N)}(\omega,A)$ is bounded due to Remark~\ref{Rm:alphaneu} (iii), which is important for \eqref{inequbirkhoff}.
	
% 	The proof for $C^{\kappa,(N)}(\omega,A)$ can be done analogously, with the set $\Omega_{C^\kappa}$, but with a different estimate instead of \eqref{inequbirkhoff}). The new estimate can be derived, using $C^\kappa(\omega)>0$, as follows:
% 	\begin{align*}
% 	C^{\kappa,(N)}(\omega,A)&=\dfrac{1}{|NA\cap\mathbb{Z}|}\sum_{i\in NA\cap\mathbb{Z}}C^\kappa(\tau_i\omega)\geq \dfrac{1}{|NA\cap\mathbb{Z}|}\sum_{i\in NB\cap\mathbb{Z}}C^\kappa(\tau_i\omega)\\[2mm]
% 	&=\dfrac{|NB\cup\mathbb{Z}|}{|NA\cup\mathbb{Z}|} C^{\kappa,(N)}(\omega,B).
% 	\end{align*}
	
% 	Defining $\Omega':=\Omega_{\alpha^{-1}}\cap\Omega_{C^\kappa}$ yields the assertion of the proposition.
% \end{proof}

\subsection{Proof of Theorem~\ref{Thm:compactnessgammarescaled}} \label{sec:4.1}

The following proof is inspired by corresponding proofs in the deterministic setting  \cite{BraidesLewOrtiz2006,ScardiaSchloemerkemperZanini2012} and  essentially follows \cite{Lauerbach-Diss}. The proof mainly relies on the uniform harmonic approximation in (LJ4), that holds true for all $\omega\in\Omega$. 

\begin{proof}[Proof of Theorem~\ref{Thm:compactnessgammarescaled}]
	Let $(v_n)$ be a sequence with $\sup_nE_n^{\gamma_n}(\omega,v_n)<+\infty$. Then, we have $v_n\in \mathcal{A}_n^{\gamma_n}(0,1)$ with $v_n(0)=0$ and $v_n(1)=\gamma_n$. By (LJ4), there exist constants $K_1,K_2>0$ such that
	\begin{align}
	\label{energieabschaetzungmitK}
	\begin{split}
	E_n^{\gamma_n}(\omega,v_n)&=\sum_{i=0}^{n-1}\left(J\left(\tau_i\omega,\dfrac{v_n^{i+1}-v_n^{i}}{\sqrt{\lambda_n}}+\delta(\tau_i\omega)\right)-J\left(\tau_i\omega,\delta(\tau_i\omega)\right) \right)\\[1mm]
	&\geq\sum_{i=0}^{n-1}\left(K_1\left(\dfrac{v_n^{i+1}-v_n^{i}}{\sqrt{\lambda_n}} \right)^2\wedge K_2\right)=\sum_{i=0}^{n-1}\left(\lambda_nK_1\left(\dfrac{v_n^{i+1}-v_n^{i}}{\lambda_n} \right)^2\wedge K_2\right).
	\end{split}
	\end{align}
	We frequently make use of this inequality in the following. First of all, one can extract from \eqref{energieabschaetzungmitK} a bound for the gradients. The energy is equibounded and all terms in the sum are positive. Together with the superlinear growth at zero due to (LJ1) this yields
	\begin{align*}
	\delta(\tau_i\omega)+\dfrac{v_n^{i+1}-v_n^{i}}{\sqrt{\lambda_n}}\geq 0,
	\end{align*}
	for all $i\in\{0,\ldots,n-1\}$ and for all $n\in\mathbb{N}$. Since we have $\delta(\tau_i\omega)\leq d$ for all $i\in\{0,\ldots,n-1\}$ due to (LJ2), we have
	\begin{align}
	\label{abschaetzungM}
	\dfrac{v_n^{i+1}-v_n^{i}}{\lambda_n}\geq -\dfrac{d}{\sqrt{\lambda_n}}.
	\end{align}
	
	\noindent
	\textbf{Step 1:} We show $\sup_n\lVert v_n\rVert_{W^{1,1}(0,1)}<+\infty$ and the existence of a subsequence $(v_{n_k})$ and $v\in BV^{\gamma}(0,1)$ such that $v_{n_k}\rightharpoonup^*v$ in $BV(0,1)$. To this end, we define
	\begin{align*}
	I_n^-&:=\left\{i\in\{0,\ldots,n-1\}\, :\,v_n^{i+1}<v_n^{i}\right\},\\[2mm]
	I_n^{--}&:=\left\{i\in I_n^-\, :\,\lambda_nK_1\left(\dfrac{v_n^{i+1}-v_n^{i}}{\lambda_n} \right)^2\geq K_2 \right\}.
	\end{align*}
	Since all addends of the rescaled energy are positive, we have
	\begin{align*}
	E_n^{\gamma_n}(\omega,v_n)&\geq\sum_{i\in I_n^-}\left(\lambda_nK_1\left(\dfrac{v_n^{i+1}-v_n^{i}}{\lambda_n} \right)^2\wedge K_2\right)=\sum_{i\in I_n^-\setminus I_n^{--}}\left(\lambda_nK_1\left(\dfrac{v_n^{i+1}-v_n^{i}}{\lambda_n}\right)^2 \right)+K_2\#I_n^{--}.
	\end{align*}
	As the energy is equibounded and $K_2>0$ this shows $I^{--}:=\sup_n\#I_n^{--}<+\infty$. Defining $(v_n')^-:=-(v_n'\wedge0)$, we get with the Hölder inequality
	\begin{align*}
	\lVert(v_n')^-\rVert_{L^1(0,1)}&=\sum_{i\in I_n^-}\lambda_n\left|\dfrac{v_n^{i+1}-v_n^{i}}{\lambda_n}\right|\stackrel{\eqref{abschaetzungM}}{\leq}\sum_{i\in I_n^-\setminus I_n^{--}}\lambda_n\left|\dfrac{v_n^{i+1}-v_n^{i}}{\lambda_n}\right|+\#I_n^{--}\lambda_n\left|\dfrac{d}{\sqrt{\lambda_n}} \right|\\[2mm]
	&\leq\left(\sum_{i\in I_n^-\setminus I_n^{--}}\lambda_n\left|\dfrac{v_n^{i+1}-v_n^{i}}{\lambda_n}\right|^2\right)^{\frac{1}{2}}\cdot\left(\sum_{i\in I_n^-\setminus I_n^{--}}\lambda_n \right)^{\frac{1}{2}}+\#I_n^{--}\sqrt{\lambda_n}d\\[2mm]
	&\leq\left(\dfrac{1}{K_1}E_n^{\gamma_n}(\omega,v_n) \right)^{\frac{1}{2}}+I^{--}d.
	\end{align*}
	Therefore, we have $\lVert(v_n')^-\rVert_{L^1(0,1)}<C$ for all $n\in\mathbb{N}$ and for a constant $C>0$, because $I^{--}<\infty$. Together with the boundary data $v_n(0)=0$ and $v_n(1)=\gamma_n$, this leads to
	\begin{align*}
	\int_{\{v_n'\geq0\}}v_n'(x)\,\mathrm{d}x=\gamma_n-\int_{\{v_n'<0\}}v_n'(x)\,\mathrm{d}x\leq\gamma_n+C.
	\end{align*}
	Hence $\lVert v_n'\rVert_{L^1(0,1)}\leq \gamma_n+2C\leq\tilde{C}$
	as $\gamma_n$ is bounded by \eqref{bedingunggamma}. Since we have $v_n(0)=0$, the Poincaré inequality now provides $\sup_n\lVert v_n\rVert_{W^{1,1}(0,1)}<+\infty$. The equiboundedness of the $W^{1,1}$-norm then again yields the existence of a subsequence $(v_{n_k})$ and $v\in BV(0,1)$ such that $v_{n_k}\rightharpoonup^*v$ in $BV(0,1)$. To obtain the boundary values, we define an extension of the sequence $v_n$ by
	\begin{align*}
	\tilde{v}_{n}^i =\begin{cases}
	0&\quad\text{if}\ i\leq0,\\
	v_n^{i}&\quad\text{if}\ 0<i< n,\\
	\gamma_n&\quad\text{if}\ i\geq n,
	\end{cases}
	\end{align*}
	which is in $W^{1,\infty}(\mathbb{R})$ because it holds $v_n(0)=0$ and $v_n(1)=\gamma_n$ for every $n\in\mathbb{N}$. Then, $\tilde{v}_{n_k}$ converges weakly$^*$ in $BV_{loc}(\mathbb{R})$ to the extension $\tilde{v}$ of $v$. Therefore, we have
	$v(0^-)=\lim\limits_{t\rightarrow 0^-}\tilde{v}(t)=0$ and  $v(1^+)=\lim\limits_{t\rightarrow1^+}\tilde{v}(t)=\gamma$,
	and thus $v\in BV^\gamma(0,1)$.\\
	
	\noindent
	\textbf{Step 2:} To show that $v\in SBV^{\gamma}(0,1)$, $v'\in L^2(0,1)$ and $\#S_v<+\infty$, we define the set
	\begin{align*}
	I_n:=\left\{i\in\{0,\ldots,n-1\}\, :\,\lambda_nK_1\left(\dfrac{v_n^{i+1}-v_n^{i}}{\lambda_n} \right)^2\geq K_2 \right\},
	\end{align*}
	and $(\tilde{v}_n)\subset SBV(0,1)$ by $\tilde{v}_n(1):=\gamma_n$ and
	\begin{align*}
	\tilde{v}_n(x):=\begin{cases}
	v_n(x)&\quad\text{if}\ x\in\lambda_n[i,i+1),\ i\notin I_n,\\[1mm]
	v_n(i\lambda_n)&\quad\text{if}\ x\in\lambda_n[i,i+1),\ i\in I_n,\ |v_n(i\lambda_n)|< |v_n((i+1)\lambda_n)|,\\[1mm]
	v_n((i+1)\lambda_n)&\quad\text{if}\ x\in\lambda_n[i,i+1),\ i\in I_n,\ |v_n(i\lambda_n)|\geq |v_n((i+1)\lambda_n)|.
	\end{cases}
	\end{align*}
	The construction of $\tilde{v}_n$ is done in such a way that we can show (i) $\lim\limits_{n\rightarrow\infty}\lVert \tilde{v}_n-v_n\rVert_{L^1(0,1)}=0$ and (ii) $\lVert\tilde{v}_n\rVert_{BV(0,1)}\leq C\lVert v_n\rVert_{W^{1,1}(0,1)}$ and therefore $\tilde{v}_n\rightharpoonup^*v$ in $BV(0,1)$ holds true up to the subsequence $v_{n_k}$, cf.~step 1 (not relabelled). We start with (i) and observe
	\begin{align}
	\label{tildev}
	\begin{split}
	\lVert \tilde{v}_n-v_n\rVert_{L^1(0,1)}&=\sum_{i\in I_n}\int_{i\lambda_n}^{(i+1)\lambda_n}\left|\tilde{v}_n(x)-v_n(x) \right|\,\mathrm{d}x\\[1mm]
	&=\sum_{i\in I_n}\int_{i\lambda_n}^{(i+1)\lambda_n}\left|\tilde{v}_n(i\lambda_n)-v_n(i\lambda_n)+\int_{i\lambda_n}^{x}\tilde{v}_n'(y)-v_n'(y) \,\mathrm{d}y\right|\,\mathrm{d}x.
	\end{split}
	\end{align}
	We have to distinguish two cases, namely $|v_n(i\lambda_n)|<|v_n((i+1)\lambda_n)|$ and $|v_n(i\lambda_n)|\geq |v_n((i+1)\lambda_n)|$. For the first one, $\tilde{v}_n(i\lambda_n)=v_n(i\lambda_n)$ holds true as well as $\tilde{v}_n'(y)\equiv0$ and therefore
	\begin{align*}
	\left|\tilde{v}_n(i\lambda_n)-v_n(i\lambda_n)+\int_{i\lambda_n}^{x}\tilde{v}_n'(y)-v_n'(y) \,\mathrm{d}y\right|&=\left|\int_{i\lambda_n}^{x}-v_n'(y) \,\mathrm{d}y\right| %\\[1mm]
	%&\leq\int_{i\lambda_n}^{x}\left| v_n'(y)\right| \,\mathrm{d}y
	\leq\int_{i\lambda_n}^{(i+1)\lambda_n}\left| v_n'(y)\right| \,\mathrm{d}y.
	\end{align*}
	For the second case $|v_n(i\lambda_n)|\geq |v_n((i+1)\lambda_n)|$, it holds true that $\tilde{v}_n(i\lambda_n)=v_n((i+1)\lambda_n)$ and $\tilde{v}_n'(y)\equiv0$, thus we get	 
	\begin{align*}
	&\left|\tilde{v}_n(i\lambda_n)-v_n(i\lambda_n)+\int_{i\lambda_n}^{x}\tilde{v}_n'(y)-v_n'(y) \,\mathrm{d}y\right| %=\left|v_n((i+1)\lambda_n)-v_n(i\lambda_n)-\int_{i\lambda_n}^{x}v_n'(y) \,\mathrm{d}y\right|
	=\left|\int_{i\lambda_n}^{(i+1)\lambda_n}v_n'(y) \,\mathrm{d}y-\int_{i\lambda_n}^{x}v_n'(y) \,\mathrm{d}y\right|\\[2mm]
	&=\left|\int_{x}^{(i+1)\lambda_n}v_n'(y) \,\mathrm{d}y\right|
	\leq\int_{x}^{(i+1)\lambda_n}\left|v_n'(y)\right| \,\mathrm{d}y\leq\int_{i\lambda_n}^{(i+1)\lambda_n}\left|v_n'(y)\right| \,\mathrm{d}y,
	\end{align*}
	which is the same result as for the first case. Therefore, we continue with \eqref{tildev} as	 
	\begin{align*}
	\lVert \tilde{v}_n-v_n\rVert_{L^1(0,1)}&\leq\sum_{i\in I_n}\int_{i\lambda_n}^{(i+1)\lambda_n}\int_{i\lambda_n}^{(i+1)\lambda_n}\left|\dfrac{v_n^{i+1}-v_n^{i}}{\lambda_n} \right|\,\mathrm{d}y\,\mathrm{d}x\\[2mm]
	&=\lambda_n\int_{0}^{1}\left|\dfrac{v_n^{i+1}-v_n^{i}}{\lambda_n} \right|\,\mathrm{d}x=\lambda_n\lVert v_n'\rVert_{L^1(0,1)}\leq\lambda_n\tilde{C},
	\end{align*}
	which shows (i) $\lim\limits_{n\rightarrow\infty}\lVert \tilde{v}_n-v_n\rVert_{L^1(0,1)}=0$. Next, we show (ii). It holds true that
	\begin{align*}
	\lVert\tilde{v}_n\rVert_{BV(0,1)}&=\int_{0}^{1}\left|\tilde{v}_n(x) \right|\,\mathrm{d}x+\int_{0}^{1}\left|\tilde{v}_n'(x) \right|\,\mathrm{d}x+\sum_{i\in I_n}\left|v_n((i+1)\lambda_n)-v_n(i\lambda_n) \right|\\[1mm]
	&\hspace{-0.4mm}\stackrel{(*)}{\leq} 2\int_{0}^{1}\left|v_n(x) \right|\,\mathrm{d}x+\sum_{i\notin I_n}\lambda_n\left|\dfrac{v_n^{i+1}-v_n^{i}}{\lambda_n}\right|+\sum_{i\in I_n}\lambda_n\left|\dfrac{v_n^{i+1}-v_n^{i}}{\lambda_n}\right|\leq C\lVert v_n\rVert_{W^{1,1}(0,1)}.
	\end{align*}
	The estimate $(*)$ can be seen as follows. For $i\notin I_n$, we have that $\tilde{v}_n(x)=v_n(x)$ and the estimate is obviously true. For $i\in I_n$, we have to distinguish between two cases, (a) $v_n(i\lambda_n)$ and $v_n((i+1)\lambda_n)$ have the same sign and (b) $v_n(i\lambda_n)$ and $v_n((i+1)\lambda_n)$ have different signs. For (a), it holds true $|\tilde{v}_n(x)|\leq|v_n(x)|$, by construction. For (b), we have $\int_{i\lambda_n}^{(i+1)\lambda_n}|\tilde{v}_n(x)|\leq 2\int_{i\lambda_n}^{(i+1)\lambda_n}|v_n(x)|$, recalling that $v_n$ is affine on the given interval. Altogether, this shows (ii) $\lVert\tilde{v}_n\rVert_{BV(0,1)}\leq C\lVert v_n\rVert_{W^{1,1}(0,1)}$.\\
	
	Moreover, $\#S_{\tilde{v}_n}=\#I_n$ by the definition of $\tilde{v}_n$. From \eqref{energieabschaetzungmitK} and with $C>0$, we get
	\begin{align*}
	%\label{eq:l2schranke}
	C&>E_n^{\gamma_n}(\omega,v_n)\geq\sum_{i\notin I_n}\left(\lambda_n K_1\left(\dfrac{v_n^{i+1}-v_n^{i}}{\lambda_n}\right)^2\right)+K_2\#I_n \\
	&\geq\min\{K_1,K_2\}\left(\int_{0}^{1}\left|\tilde{v}_n'(x) \right|^2\,\mathrm{d}x+\#S_{\tilde{v}_n} \right),
	\end{align*}
	which yields $\sup_n\lVert \tilde{v}_n'\rVert_{L^2(0,1)}<+\infty$ and $\sup_n\#S_{\tilde{v}_n}<+\infty$. Therefore, the closure theorem for SBV \cite[Theorem~4.7]{AmbrosioFuscoPallara} provides $v\in SBV(0,1)$, $\tilde{v}_n'\rightharpoonup v'$ in $L^1(0,1)$, up to the subsequence $v_{n_k}$, cf.~step 1 (not relabelled), the weak$^*$-convergence of the jump part of the derivative $D^j\tilde{v}_n\rightharpoonup^* D^jv$ in $(0,1)$ and $\#S_v\leq\liminf_{n\rightarrow\infty}\#S_{\tilde{v}_n}<\infty$. Further, $\sup_n\lVert \tilde{v}_n'\rVert_{L^2(0,1)}<+\infty$ yields $\tilde{v}_n'\rightharpoonup v'$ in $L^2(0,1)$ with $v'\in L^2(0,1)$. By step 1, we have $v\in BV^{\gamma}(0,1)$, which also provides $v\in SBV^{\gamma}(0,1)$. This completes step~2.\\
	
	\noindent
	\textbf{Step 3:} We show that there exists a finite set $S\subset[0,1]$ such that $v_{n_k}\rightharpoonup v$ locally weakly in $H^1((0,1)\setminus S)$.
	
	In order to simplify notation, we omit the index $k$ of the subsequence. Since $\sup_n\#S_{\tilde{v}_n}<\infty$, there exist $m\in\mathbb{N}$ and $x_1^n,\ldots,x_m^n\in[0,1]$ such that $S_{\tilde{v}_n}\subset\{x_i^n:i\in\{1,\ldots,m\}\}$. From $D^j\tilde{v}_n\rightharpoonup^* D^jv$ we get that $x_i^n\rightarrow x_i\in[0,1]$ for all $i\in\{1,\ldots,m\}$ up to a subsequence. For a fixed $\eta>0$, we define $S:=\{x_1,\ldots,x_m\}$ and $S_{\eta}:=\bigcup_{i=1}^m(x_i-\eta,x_i+\eta)$. Due to the convergence of $x_i^n$, there exists $N\in\mathbb{N}$ such that $S_{\tilde{v}_n}\subset S_{\eta}$ and therefore $v_n\equiv \tilde{v}_n$ on $(0,1)\setminus S_{\eta}$ for $n\geq N$ and $\sup_{n\geq N}\lVert v_n'\rVert_{L^2((0,1)\setminus S_{\eta})}<+\infty$. Then, the Poincaré inequality on every connected subset $A$ of $(0,1)\setminus S_{\eta}$ yields $\sup_n\lVert v_n\rVert_{L^2((0,1)\setminus S_{\eta})}<\infty$, which can be shown as follows:
	\begin{align*}
	\int_{A}\left|v_n \right|^2\,\mathrm{d}x&\leq\int_{A}\left((v_n-\bar{v}_n)^2+2v_n\bar{v}_n \right)\,\mathrm{d}x\leq C\lVert v_n'\rVert_{L^2(0,1)}+2\int_{A}v_n\,\mathrm{d}x\cdot \bar{v}_n\\[1mm]
	&\leq C\lVert v_n'\rVert_{L^2(0,1)}+\frac{2}{|A|}\lVert v_n\rVert^2_{L^1(0,1)},
	\end{align*}
	where $\bar{v}_n:=\frac{1}{|A|}\int_{A}v_n(x)\,\mathrm{d}x$. The right hand side is uniformly bounded, which was shown in step 1 and step 2. Altogether, we have $v_n\rightharpoonup v$ in $H^1((0,1)\setminus S_{\eta})$. Since $\eta$ was chosen arbitrary, we get the asserted result by passing to the limit as $\eta\rightarrow 0$.\\
	
	\noindent
	\textbf{Step 4:} We show $[v]\geq 0$ in $[0,1]$, i.e.~$[v](x)>0$ for $x\in S_v$.
	
	Following \cite{BraidesLewOrtiz2006,ScardiaSchloemerkemperZanini2012}, we observe that there exist constants $D_1, D_2, D_3>0$ such that
	\begin{align}
	\label{energiePhi}
	E_n^{\gamma_n}(\omega,v_n)=\sum_{i=0}^{n-1}\left(J\left(\tau_i\omega,\dfrac{v_n^{i+1}-v_n^{i}}{\sqrt{\lambda_n}}+\delta(\tau_i\omega)\right)-J(\tau_i\omega,\delta(\tau_i\omega)) \right)\geq \sum_{i=0}^{n-1}\varphi\left(\dfrac{v_n^{i+1}-v_n^{i}}{\sqrt{\lambda_n}} \right)
	\end{align}
	with
	\begin{align*}
	\varphi(x):=\begin{cases}
	D_1x^2\wedge D_2&\quad\text{for}\ x>0,\\
	D_1x^2\wedge D_3&\quad\text{for}\ x\leq 0.
	\end{cases}
	\end{align*}
	It is not restrictive to assume
	\begin{align}
	\label{abschaetzungD}
	D_3>D_1d^2,
	\end{align}
	because of the superlinear growth at $z\to0^+$ of the potentials $J(\omega,z)$ and the asymptotic behaviour $\lim_{z\to+\infty}J(\omega,z)=0>J(\omega,\delta(\omega))$. We define
	\begin{align*}
	\tilde{I}_n^+&:=\left\{i\in\{0,\ldots,n-1\}\, :\, v_n^{i+1}>v_n^{i}\ \text{and}\ D_1\left(\dfrac{v_n^{i+1}-v_n^{i}}{\sqrt{\lambda_n}}\right)^2>D_2  \right\},\\[1mm]
	\tilde{I}_n^-&:=\left\{i\in\{0,\ldots,n-1\}\, :\, v_n^{i+1}<v_n^{i}\ \text{and}\ D_1\left(\dfrac{v_n^{i+1}-v_n^{i}}{\sqrt{\lambda_n}}\right)^2>D_3  \right\}.
	\end{align*}
	First of all, $\tilde{I}_n^-=\emptyset$ is valid, because
	\begin{align*}
	(v_n^{i+1}-v_n^{i})^2\stackrel{\in\tilde{I}_n^-}{>}\dfrac{D_3}{D_1}\lambda_n\stackrel{\eqref{abschaetzungD}}{>}\dfrac{D_1d^2}{D_1}\lambda_n=d^2\lambda_n\stackrel{\in\tilde{I}_n^-\text{ and }\eqref{abschaetzungM}}{\geq}(v_n^{i+1}-v_n^{i})^2
	\end{align*}
	is a contradiction. Since all summands in \eqref{energiePhi} are positive, we get
	\begin{align*}
	E_n^{\gamma_n}(\omega,v_n)\geq\sum_{i\notin \tilde{I}_n^{\pm}}\varphi\left(\dfrac{v_n^{i+1}-v_n^{i}}{\sqrt{\lambda_n}} \right)+\#\tilde{I}_n^+D_2.
	\end{align*}
	Thus, the equiboundedness of the energy implies $\sup_n\#\tilde{I}_n^+<\infty$. We define $(\hat{v}_n)\subset SBV(0,1)$ as
	\begin{align*}
	\hat{v}_n(x):=\begin{cases}
	v_n(x)&\quad\text{for}\ x\in\lambda_n[i,i+1),\ i\notin\tilde{I}_n^+,\\[1mm]
	v_n(i\lambda_n)&\quad\text{for}\ x\in\lambda_n[i,i+1),\ i\in\tilde{I}_n^+.
	\end{cases}
	\end{align*}
	Similarly as in step~2, we have $\hat{v}_n\rightharpoonup^* v$ in $BV(0,1)$, up to a subsequence. By definition, $\hat{v}_n(x)$ has only positive jumps, i.e.~$D^j\hat{v}_n\geq0$ in $(0,1)$. Now, we define the following auxiliary functions in $SBV(a,b)$ for any $a<0$ and $b>1$.
	\begin{align*}
	w(x):=\begin{cases}
	0&\quad\text{for}\ x\leq0,\\
	v(x)&\quad\text{for}\ 0<x<1,\\
	\gamma&\quad\text{for}\ 1\leq x,
	\end{cases}\qquad
	w_n(x):=\begin{cases}
	0&\quad\text{for}\ x\leq0,\\
	\hat{v}_n(x)&\quad\text{for}\ 0<x<1,\\
	\gamma_n&\quad\text{for}\ 1\leq x,
	\end{cases}
	\end{align*}
	to capture also possible jumps at the boundary. With $\hat{v}_n\rightharpoonup^* v$ in $BV(0,1)$, this also yields $w_n\rightharpoonup^* w$ in $BV(a,b)$ for any $a<0$ and $1<b$.
	\begin{align*}
	E_n^{\gamma_n}(\omega,v_n)&\geq\sum_{i=0}^{n-1}\varphi\left(\dfrac{v_n^{i+1}-v_n^{i}}{\sqrt{\lambda_n}} \right)=
	\sum_{i\notin \tilde{I}_n^{\pm}}\lambda_nD_1\left(\dfrac{\hat{v}_n^{i+1}-\hat{v}_n^{i}}{\lambda_n} \right)^2+D_2\#S_{\hat{v}_n}\\[1mm]
	&\geq D_1\int_{0}^{1}\left|\hat{v}_n'(x) \right|^2\,\mathrm{d}x+D_2\#S_{\hat{v}_n}=D_1\int_{0}^{1}\left|w_n'(x) \right|^2\,\mathrm{d}x+D_2\#S_{w_n}.
	\end{align*}
	Again due to the SBV closure theorem \cite[Theorem~4.7]{AmbrosioFuscoPallara}, this provides $D^jw_n\rightharpoonup^*D^jw$ in $(a,b)$. By construction of $w_n$ and by $D^j\hat{v}_n\geq0$ in $(0,1)$, we get $D^jw_n\geq0$ in $(a,b)$. Altogether, this yields $D^jw\geq0$ in $(a,b)$. And since $D^jv$ is the restriction of $D^jw$ to $[0,1]$, we finally obtain $D^jv\geq 0$ in $[0,1]$.
\end{proof}

\subsection{Proof of Theorem~\ref{Thm:rescaled}}
\label{Sec:GammaLimitRescaled}

Before proving the main theorem of this article, we first state  and show an extended version of the Birkhoff ergodic theorem \cite[Theorem~2.3]{Krengel}.

\begin{proposition}
	\label{Prop:ergodicbeta}
	Let $(\Omega,\mathcal{F},\mathbb{P})$ be a probability space and let $(\tau_i)_{i\in\mathbb{Z}}$ be an additive, stationary and ergodic group action. For $\epsilon>0$ and $x\in\mathbb{R}$, let $I_x^{\epsilon}=]x-\epsilon,x+\epsilon[$ and let $f$ be an integrable random variable. Then there exists an $\Omega'\subset\Omega$ with $\mathbb{P}(\Omega')=1$ such that for every $x\in\mathbb{R}$ and every $k\in\mathbb{Q}$ the following limit exists: 
	\begin{align*}
	\lim\limits_{n\rightarrow\infty}\dfrac{1}{2\epsilon n}\sum_{i\in\mathbb{Z}\cap n I_x^{\epsilon}}\chi_{(f(\tau_i\omega)\leq k)}=\mathbb{E}\left[\chi_{f\leq k} \right].
	\end{align*}
\end{proposition}
\begin{proof}
	From the Birkhoff ergodic theorem \cite[Theorem~2.3]{Krengel}, we get the existence of $\Omega_{x,k,\epsilon}\subset\Omega$ with $\mathbb{P}(\Omega_{x,k,\epsilon})=1$ such that
	\begin{align*}
	\lim\limits_{n\rightarrow\infty}\dfrac{1}{2\epsilon n}\sum_{i\in\mathbb{Z}\cap n I_x^{\epsilon}}\chi_{(f(\tau_i\omega)\leq k)}=\mathbb{E}\left[\chi_{f\leq k} \right]
	\end{align*}
	for a fixed $x\in\mathbb{Q}$, fixed $k\in\mathbb{Q}$ and a fixed $\epsilon\in\mathbb{Q}^+$. Since for $\tilde{\Omega}:=\bigcap_{x\in\mathbb{Q},k\in\mathbb{Q},\epsilon\in\mathbb{Q}}\Omega_{x,k,\epsilon}$, we have $\mathbb{P}(\tilde{\Omega})=1$, the assertion of the theorem is already shown for every $x\in\mathbb{Q}$, $k\in\mathbb{Q}$ and $\epsilon\in\mathbb{Q}^+$. It remains to expand it to $\epsilon\in\mathbb{R}^+$ and $x\in\mathbb{R}$, which is done in two steps.\\
	
	\noindent
	\textbf{Step 1:} We prove the assertion for $\epsilon\in\mathbb{R}^+\setminus\mathbb{Q}^+$. For  this, notice that for every $\epsilon\in\mathbb{R}^+\setminus\mathbb{Q}^+$ there exist sequences $\left(\epsilon_N^1\right)\subset\mathbb{Q}^+$ and $\left(\epsilon_N^2\right)\subset\mathbb{Q}^+$ with $\epsilon_N^1\to \epsilon$, $\epsilon_N^2\to\epsilon$ and $\epsilon_N^1\leq\epsilon\leq\epsilon_N^2$ for every $N\in\mathbb{N}$. The definition implies $I_{x}^{\epsilon_N^1}\subset I_{x}^{\epsilon}\subset I_{x}^{\epsilon_N^2}$. Therefore, it holds true that
	\begin{align*}
	\dfrac{2\epsilon_N^1 n}{2\epsilon n}\dfrac{1}{2\epsilon_N^1 n}\sum_{i\in\mathbb{Z}\cap n I_{x}^{\epsilon_N^1}}\chi_{(f(\tau_i\omega)\leq k)}\leq\dfrac{1}{2\epsilon n}\sum_{i\in\mathbb{Z}\cap n I_{x}^{\epsilon}}\chi_{(f(\tau_i\omega)\leq k)}\leq \dfrac{2\epsilon_N^2 n}{2\epsilon n}\dfrac{1}{2\epsilon_N^2 n}\sum_{i\in\mathbb{Z}\cap n I_{x}^{\epsilon_N^2}}\chi_{(f(\tau_i\omega)\leq k)}.
	\end{align*}
	Taking the $\limsup_{n\to\infty}$ on both sides and recalling $x,k,\epsilon_N^1,\epsilon_N^2\in\mathbb{Q}^+$ , we get
	\begin{align*}
	\dfrac{\epsilon_N^1 }{\epsilon }\mathbb{E}[\chi_{f\leq k}]\leq\limsup_{n\to\infty}\dfrac{1}{2\epsilon n}\sum_{i\in\mathbb{Z}\cap n I_{x}^{\epsilon}}\chi_{(f(\tau_i\omega)\leq k)}\leq \dfrac{\epsilon_N^2 }{\epsilon }\mathbb{E}[\chi_{f\leq k}].
	\end{align*}
	Passing subsequently to the limit as $N\to\infty$, we obtain the assertion by $\frac{\epsilon_N^{1,2}}{\epsilon}\to1$.\\
	
	\noindent
	\textbf{Step 2:} To prove the assertion for $x\in\mathbb{R}\setminus\mathbb{Q}$, we notice that for every $x_0\in\mathbb{R}\setminus\mathbb{Q}$ there exist sequences $(x_N)\subset\mathbb{Q}$ and $(\epsilon_N)\subset\mathbb{R}$ with $x_N\rightarrow x_0$ and $\epsilon_N=\epsilon+|x_N-x_0|$ for every $N$. The definition implies $I_{x_0}^{\epsilon}\subset I_{x_N}^{\epsilon_N}$. Therefore, we get
	\begin{align*}
	\dfrac{1}{2\epsilon n}\sum_{i\in\mathbb{Z}\cap n I_{x_0}^{\epsilon}}\chi_{(f(\tau_i\omega)\leq k)}\leq\dfrac{2\epsilon_N n}{2\epsilon n}\dfrac{1}{2\epsilon_Nn}\sum_{i\in\mathbb{Z}\cap n I_{x_N}^{\epsilon_N}}\chi_{(f(\tau_i\omega)\leq k)}.
	\end{align*}
	Taking $\limsup_{n\rightarrow\infty}$, we get
	\begin{align*}
	\limsup_{n\rightarrow\infty}\dfrac{1}{2\epsilon n}\sum_{i\in\mathbb{Z}\cap n I_{x_0}^{\epsilon}}\chi_{(f(\tau_i\omega)\leq k)}\leq\dfrac{\epsilon_N }{\epsilon}\mathbb{E}\left[\chi_{f\leq k} \right],
	\end{align*}
	because $x_N\in\mathbb{Q}$. Subsequently, we take the limit as $N$ tends to infinity and get
	\begin{align}
	\label{betasup}
	\limsup_{n\rightarrow\infty}\dfrac{1}{2\epsilon n}\sum_{i\in\mathbb{Z}\cap n I_{x_0}^{\epsilon}}\chi_{(f(\tau_i\omega)\leq k)}\leq\mathbb{E}\left[\chi_{f\leq k} \right].
	\end{align}
	Analogously, to prove
	\begin{align}
	\label{betainf}
	\liminf_{n\rightarrow\infty}\dfrac{1}{2\epsilon n}\sum_{i\in\mathbb{Z}\cap n I_{x_0}^{\epsilon}}\chi_{(f(\tau_i\omega)\leq k)}\geq\mathbb{E}\left[\chi_{f\leq k} \right],
	\end{align}
    we replace the requirement $I_{x_0}^{\epsilon}\subset I_{x_N}^{\epsilon_N}$ by $I_{x_N}^{\epsilon_N}\subset I_{x_0}^{\epsilon}$ and $\epsilon_N=\epsilon+|x_N-x_0|$ by $\epsilon_N=\epsilon-|x_N-x_0|$. Then, \eqref{betasup} and \eqref{betainf} together yield
	\begin{align*}
	\lim_{n\rightarrow\infty}\dfrac{1}{2\epsilon n}\sum_{i\in\mathbb{Z}\cap n I_{x_0}^{\epsilon}}\chi_{(f(\tau_i\omega)\geq k)}=\mathbb{E}\left[\chi_{f\leq k} \right]
	\end{align*}
	for $x_0\in\mathbb{R}\setminus\mathbb{Q}$, which completes the proof.
\end{proof}

Finally, we show the $\Gamma$-limit result of the rescaled energy functional, the main result of this article.

\begin{proof}[Proof of Theorem~\ref{Thm:rescaled}]
	First of all, we notice that the expectation value of $\alpha^{-1}$ exists due to Remark~\ref{Rm:integrabilityandexpectationvalues2} and therefore $\underline{\alpha}$ is well defined.
	In the following, we  several times use a Taylor expansion of $J(\tau_i\omega,x)$, which reads
	\begin{align}
	\label{taylor}
	J(\tau_i\omega,\delta(\tau_i\omega)+x)=J(\tau_i\omega,\delta(\tau_i\omega))+\alpha(\tau_i\omega)x^2+\eta(\tau_i\omega,x),
	\end{align}
	with $\alpha(\tau_i\omega):=\frac{1}{2}J''(\tau_i\omega,\delta(\tau_i\omega))$ and $\tfrac{\eta(\tau_i\omega,x)}{|x|^2}\rightarrow0\quad\text{as}\ |x|\rightarrow0$.
	Since $J\in C^3$ by (LJ1), we can use the Lagrange form of the remainder and get
	\begin{align}
	\label{lagrangeremainder}
	\eta(\tau_i\omega,x)=\frac{1}{6}\left.\dfrac{\partial^3 J(\tau_i\omega,y)}{\partial y^3}\right|_{y=\xi}x^3\quad\text{for some}\ \xi\ \text{between}\ \delta(\tau_i\omega)\ \text{and}\ \delta(\tau_i\omega)+x.
	\end{align}
	
	\noindent
	\textbf{Step 1.} Liminf inequality.
	
	Let $v\in L^1(0,1)$ and let $(v_n)\subset L^1(0,1)$ be a sequence with $v_n\rightarrow v$ in $L^1(0,1)$. We have to show
	\begin{align}
	\label{liminfskaliert}
	\liminf_{n\rightarrow\infty}E_n^{\gamma_n}(\omega,v_n)\geq \underline{\alpha}\int_{0}^{1}\left|v'(x) \right|^2\,\mathrm{d}x+\beta\#S_v.
	\end{align}
	Without loss of generality we consider a (sub-)sequences with  $\sup_nE_n^{\gamma_n}(\omega,v_n)<\infty$. For such a sequence, the compactness result from Theorem~\ref{Thm:compactnessgammarescaled} provides $v\in SBV_c^{\gamma}(0,1)$ and the existence of a finite set $S=\{x_1,\ldots,x_N\}$ such that $v_n\rightharpoonup v$ locally weakly in $H^1((0,1)\setminus S)$. Now, let $\rho>0$ be such that $|x_i-x_j|>2\rho$ for all $x_i,x_j\in S$, $i\neq j$. We define
	\begin{align*}
	S_{\rho}&:=\bigcup_{i=1}^N(x_i-\rho,x_i+\rho),\\[1mm]
	Q_n(\rho)&:=\left\{i\in\{0,\ldots,n-1\}\, :\,(i,i+1)\lambda_n\subset(0,1)\setminus S_{\rho} \right\},\\[2mm]
	S_n(\rho)&:=\left\{i\in\{0,\ldots,n-1\}\, :\,i\notin Q_n(\rho) \right\}.
	\end{align*}
	The sets $S_n(\rho)$ and $Q_n(\rho)$ separate indices close to a jump of $v$ from those away from a jump. According to this, the energy can also be separated into
	\begin{align*}
	E_n^{\gamma_n}(\omega,v_n)=&\sum_{i\in Q_n(\rho)}\left(J\left(\tau_i\omega,\dfrac{v_n^{i+1}-v_n^{i}}{\sqrt{\lambda_n}}+\delta(\tau_i\omega) \right)-J\left(\tau_i\omega,\delta(\tau_i\omega) \right) \right)\\[1mm]
	&+\sum_{i\in S_n(\rho)}\left(J\left(\tau_i\omega,\dfrac{v_n^{i+1}-v_n^{i}}{\sqrt{\lambda_n}}+\delta(\tau_i\omega) \right)-J\left(\tau_i\omega,\delta(\tau_i\omega) \right) \right).
	\end{align*} 
	We now show that
	\begin{align}
	\label{zielliminfQ}
	\liminf_{n\rightarrow\infty}\sum_{i\in Q_n(\rho)}\left(J\left(\tau_i\omega,\dfrac{v_n^{i+1}-v_n^{i}}{\sqrt{\lambda_n}}+\delta(\tau_i\omega) \right)-J\left(\tau_i\omega,\delta(\tau_i\omega) \right) \right)\geq\underline{\alpha}\int_{(0,1)\setminus S_{\rho}}|v'(x)|^2\,\mathrm{d}x
	\end{align}
	and
	\begin{align}
	\label{zielliminfS}
	\liminf_{n\rightarrow\infty}\sum_{i\in S_n(\rho)}\left(J\left(\tau_i\omega,\dfrac{v_n^{i+1}-v_n^{i}}{\sqrt{\lambda_n}}+\delta(\tau_i\omega) \right)-J\left(\tau_i\omega,\delta(\tau_i\omega) \right) \right)\geq\beta\#S_v,
	\end{align}
	which provides \eqref{liminfskaliert} as $\rho\rightarrow0$ and by using $v'\in L^2(0,1)$.\\
	
	\noindent
	\textit{Step A: Proof of the elastic part of the energy \eqref{zielliminfQ}.}
	
	 With the help of $M\in\mathbb{N}$ we introduce a coarser scale and  define
	\begin{align*}
	I_n&:=\left\{i\in\{0,\ldots,n-1\}\, :\,\left|\dfrac{v_n^{i+1}-v_n^{i}}{\lambda_n} \right|>\lambda_n^{-\frac{1}{8}} \right\},\\[3mm]
	I_{n,M}&:=\bigg\{j\in\{0,\ldots,n-1\}\cap M\mathbb{Z}\, :\,\{j,\ldots,j+M-1\}\cap I_n\neq\emptyset \bigg\},\\[3mm]
	\chi_n(x)&:=\begin{cases}
	1&\quad\text{if}\quad x\in[i,i+1)\lambda_n\ \text{and}\ i\in \mathbb{Z}\setminus I_{n},\\
	0&\quad\text{if}\quad x\in[i,i+1)\lambda_n\ \text{and}\ i\in I_{n},
	\end{cases}\\[3mm]
	\chi_{n,M}(x)&:=\begin{cases}
	1&\quad\text{if}\quad x\in[j,j+M)\lambda_n\ \text{and}\ j\in M\mathbb{Z}\setminus I_{n,M},\\
	0&\quad\text{if}\quad x\in[j,j+M)\lambda_n\ \text{and}\ j\in I_{n,M}.
	\end{cases}
	\end{align*}
	%The sets $I_n$ and $I_{n,M}$ and the associated functions $\chi_n$ and $\chi_{n,M}$ are used in what follows. 
	The choice of the exponent $-1/8$ is of technical nature and fits to the derivation of the liminf and limsup estimates. %First, we derive some properties of these functions. 
	By the positivity of all summands and \eqref{energieabschaetzungmitK}, we obtain for $n$ large enough
	\begin{align*}
	E_n^{\gamma_n}(\omega,v_n)&\geq\sum_{i=0}^{n-1}\left(\lambda_nK_1\left(\dfrac{v_n^{i+1}-v_n^{i}}{\lambda_n} \right)^2\wedge K_2\right)\geq\sum_{i\in I_n}\left(\lambda_nK_1\left(\dfrac{v_n^{i+1}-v_n^{i}}{\lambda_n} \right)^2\wedge K_2\right)\\[2mm]
	&\geq\#I_n\lambda_nK_1\left(\lambda_n^{-\frac{1}{8}}\right)^2=\#I_nK_1\lambda_n^{\frac{3}{4}},
	\end{align*}
	which shows $\#I_n=\mathcal{O}(\lambda_n^{-\frac{3}{4}})$ due to the equiboundedness of the energy. Furthermore, we have
	\begin{align*}
	\left|\{x\in(0,1)\, :\,\chi_n(x)\neq1 \} \right|\leq\left|\{x\in(0,1)\, :\,\chi_{n,M}(x)\neq1 \} \right|\leq M\#I_n\cdot\lambda_n,
	\end{align*}
	i.e.~$\chi_n\rightarrow1$ and $\chi_{n,M}\rightarrow1$ bounded in measure in $(0,1)$ as $ n\rightarrow\infty$.

	We set $\chi_n^{i}:=\chi_n(i\lambda_n)$ and likewise for $\chi_{n,M}$. By the Taylor expansion \eqref{taylor}, the energy then reads
	\begin{align}
	\label{twoterms}
	\begin{split}
	&\sum_{i\in Q_n(\rho)}\left(J\left(\tau_i\omega,\dfrac{v_n^{i+1}-v_n^{i}}{\sqrt{\lambda_n}}+\delta(\tau_i\omega) \right)-J\left(\tau_i\omega,\delta(\tau_i\omega) \right) \right)\\[2mm]
	&= \sum_{i\in Q_n(\rho)}\left( \alpha(\tau_i\omega)\left(\dfrac{v_n^{i+1}-v_n^{i}}{\sqrt{\lambda_n}}\right)^2+\eta\left(\tau_i\omega,\dfrac{v_n^{i+1}-v_n^{i}}{\sqrt{\lambda_n}} \right) \right)\\[2mm]
	&\geq \sum_{i\in Q_n(\rho)}\left(\chi_n^{i} \alpha(\tau_i\omega)\lambda_n\left(\dfrac{v_n^{i+1}-v_n^{i}}{\lambda_n}\right)^2+\chi_n^{i}\eta\left(\tau_i\omega,\dfrac{v_n^{i+1}-v_n^{i}}{\sqrt{\lambda_n}} \right) \right).
	\end{split}
	\end{align}
	We consider both terms of the sum separately in the next two steps. This is possible because the second term vanishes as $n\rightarrow\infty$ and we can use $\liminf_{n\to\infty}(a_n+b_n)=\liminf_{n\to\infty}(a_n)+b$ for $b_n\to b$.\\

	\noindent
	\textit{Step B: Proof of the elastic part of the energy \eqref{zielliminfQ} -- Second addend of \eqref{twoterms}.} 
	
	In order to show convergence to zero of the second term on the right hand side of \eqref{twoterms}, we use the Lagrange form of the remainder from \eqref{lagrangeremainder} and get, with $\xi_i$ between $\delta(\tau_i\omega)$ and $\delta(\tau_i\omega)+\frac{v_n^{i+1}-v_n^i}{\sqrt{\lambda_n}}$,
	\begin{align*}
	\sum_{i\in Q_n(\rho)}\chi_n^{i}\eta\left(\tau_i\omega,\dfrac{v_n^{i+1}-v_n^{i}}{\sqrt{\lambda_n}} \right)&=\sum_{i\in Q_n(\rho)}\chi_n^{i}\frac{1}{6}\left.\frac{\partial^3 J(\tau_i\omega,y)}{\partial y^3}\right|_{y=\xi_i}\left(\dfrac{v_n^{i+1}-v_n^{i}}{\sqrt{\lambda_n}} \right)^3.
	\end{align*}
	For all $i\in Q_n(\rho)$ with $\chi_n^i(x)\neq0$ we have $\left|\frac{v_n^{i+1}-v_n^{i}}{\lambda_n}\right|\leq\lambda_n^{-\frac{1}{8}}$ and hence  $\left|\frac{v_n^{i+1}-v_n^{i}}{\sqrt{\lambda_n}}\right| %=\sqrt{\lambda_n}\left|\frac{v_n^{i+1}-v_n^{i}}{\lambda_n}\right|
	\leq\lambda_n^{\frac{3}{8}}$. Therefore, $\xi_i\in[\delta(\tau_i\omega),\delta(\tau_i\omega)+\kappa]\subset [\delta(\tau_i\omega)-\kappa,\delta(\tau_i\omega)+\kappa]$ for $n$ large enough with $\kappa<\kappa^*$ from (R1). Thus, we obtain for $n$ large enough
	\begin{align*}
	\sum_{i\in Q_n(\rho)}\left|\chi_n^{i}\eta\left(\tau_i\omega,\dfrac{v_n^{i+1}-v_n^{i}}{\sqrt{\lambda_n}} \right)\right|&\leq\sum_{i\in Q_n(\rho)}\frac{1}{6}\left|\left.\frac{\partial^3 J(\tau_i\omega,y)}{\partial y^3}\right|_{y=\xi_i}\right|\left(\lambda_n^{\frac{3}{8}}\right)^3\\[2mm]
	%=\sum_{i\in Q_n(\rho)}\chi_n^{i}\frac{1}{6}\left|\left.\frac{\partial^3 J(\tau_i\omega,y)}{\partial y^3}\right|_{y=\xi_i}\right|\left|\dfrac{v_n^{i+1}-v_n^{i}}{\sqrt{\lambda_n}} \right|^3\\[2mm]
	&\leq\frac{1}{6}\lambda_n^{\frac{1}{8}}\lambda_n\sum_{i=0}^{n-1}\sup_{x\in[\delta(\tau_i\omega)-\kappa,\delta(\tau_i\omega)+\kappa]}\left|\left.\frac{\partial^3 J(\tau_i\omega,y)}{\partial y^3}\right|_{y=x}\right|\leq C\lambda_n^{\frac{1}{8}},
	\end{align*} 
	where the last estimate is due to the convergence of the random variable $C^\kappa$ to its expectation value, see Proposition~\ref{Prop:averages2}. Therefore, the whole expression converges to zero as $n\to\infty$, which proves the assertion of convergence to zero of the second term to be correct.
	
	\medskip 
\begin{figure}[ht]
\centering
		\includegraphics[width=0.6\linewidth]{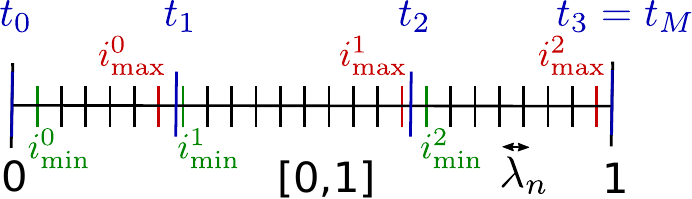}
		\caption{The definitions of $i_{\mathrm{min}}^m$ and $i_{\mathrm{max}}^m$ for $M=3$.}
	\label{fig:iminmaxreskaliert}
\end{figure}

	\noindent
	\textit{Step C: Proof of the elastic part of the energy \eqref{zielliminfQ} -- First addend of \eqref{twoterms}.} 
	
	We continue with the first part of the sum in \eqref{twoterms} and rearrange it. Following the construction in the proof of the liminf-inequality of Theorem~3.1 in \cite{unserPaper1}, we define for $\delta>0$ and $M\in\mathbb{N}$ the coarse grained grid $t_0,\ldots,t_M\in[0,1]$ with $t_0=0$, $t_M=1$ and $\delta<t_{m+1}-t_m<2\delta$ in such a way that $v_{n}(t_m)\rightarrow v(t_m)$ pointwise as $n\rightarrow\infty$ and for every $m=0,\ldots,N$. For better readability, we define $I_m:=[t_m,t_{m+1})$ for $m=0,\ldots,M$. Further, we set, cf.\ Figure~\ref{fig:iminmaxreskaliert},
	\begin{align*}
	i_{\mathrm{min}}^m:=\min\left\{i\, :\,i\in\mathbb{Z}\cap n I_m \right\}, %\\[1mm]
	\qquad i_{\mathrm{max}}^m :=\max\left\{i\, :\,i\in\mathbb{Z}\cap n I_m \right\},
	\end{align*}
	which satisfy $\lambda_ni_{\mathrm{min}}^m\rightarrow t_m$ and $\lambda_ni_{\mathrm{max}}^m\rightarrow t_{m+1}$ as $n\rightarrow\infty$ by definition.
	The energy then reads
	\begin{align*}
	&\lefteqn{\sum_{i\in Q_n(\rho)}\hspace{-2mm}\lambda_n\chi_{n}^{i} \alpha(\tau_i\omega)\left(\dfrac{v_n^{i+1}-v_n^{i}}{\lambda_n}\right)^2\geq \sum_{i\in Q_n(\rho)}\hspace{-2mm}\lambda_n\chi_{n,M}^{i} \alpha(\tau_i\omega)\left(\dfrac{v_n^{i+1}-v_n^{i}}{\lambda_n}\right)^2}\\[1mm]
	&\geq\hspace{-4mm}\sum_{\footnotesize\substack{m=0\\ (nI_m\cap\mathbb{Z} )\subset Q_n(\rho)}}^{M}\hspace{-4mm}\chi_{n,M}^{m}\sum_{i\in\mathbb{Z}\cap nI_m}\hspace{-2mm}\lambda_n \alpha(\tau_i\omega)\left(\dfrac{v_n^{i+1}-v_n^{i}}{\lambda_n}\right)^2\\[1mm]
	&\geq\hspace{-4mm}\sum_{\footnotesize\substack{m=0\\ (nI_m\cap\mathbb{Z} )\subset Q_n(\rho)}}^{M}\hspace{-4mm}\chi_{n,M}^{m}\min\left\{\sum_{i\in\mathbb{Z}\cap nI_m}\hspace{-2mm}\lambda_n \alpha(\tau_i\omega)\left(\dfrac{v_n^{i+1}-v_n^{i}}{\lambda_n}\right)^2:\sum_{i\in\mathbb{Z}\cap nI_m}\hspace{-2mm}\dfrac{v_n^{i+1}-v_n^{i}}{\lambda_n}=n\left(v_n^{i_{\mathrm{max}}^{m}+1}-v_n^{i_{\mathrm{min}}^m}\right)\right\}\\[1mm]
	&\hspace{-0.4mm}\stackrel{(*)}{=}\hspace{-4mm}\sum_{\footnotesize\substack{m=0\\ (nI_m\cap\mathbb{Z} )\subset Q_n(\rho)}}^{M}\hspace{-4mm}\chi_{n,M}^{m}\lambda_n\left(i_{\mathrm{max}}^m-i_{\mathrm{min}}^m+1\right)\left(\dfrac{1}{i_{\mathrm{max}}^m-i_{\mathrm{min}}^m+1}\sum_{i\in\mathbb{Z}\cap nI_m}\dfrac{1}{\alpha(\tau_i\omega)}\right)^{-1}\left(\dfrac{v_n^{i_{\mathrm{max}}^{m}+1}-v_n^{i_{\mathrm{min}}^m}}{\lambda_n\left(i_{\mathrm{max}}^m-i_{\mathrm{min}}^m+1\right)} \right)^2,
	\end{align*}
	where $(*)$ is obtained as a direct calculation of the minimizer and the minimum. Now, we derive the limit $\liminf\limits_{n\to\infty}$. Therefore, note that 
	\begin{align*}
	\liminf\limits_{n\rightarrow\infty}\left(\dfrac{1}{i_{\mathrm{max}}^m-i_{\mathrm{min}}^m+1}\sum_{i\in\mathbb{Z}\cap nI_m}\dfrac{1}{\alpha(\tau_i\omega)}\right)^{-1}=\left(\mathbb{E}[\alpha^{-1}]\right)^{-1}.
	\end{align*}
	This follows from Proposition~\ref{Prop:averages2}, since $i_{\mathrm{max}}^m-i_{\mathrm{min}}^m+1=|\mathbb{Z}\cap nI_m|$. As a result, we get with $\liminf_{n}(a_nb_n)\geq\liminf_n(a_n)\cdot\liminf_n(b_n)$
	\begin{align*}
	\begin{split}
	&\liminf_{n\rightarrow\infty}\sum_{i\in Q_n(\rho)}\hspace{-2mm}\lambda_n\chi_n^{i} \alpha(\tau_i\omega)\left(\dfrac{v_n^{i+1}-v_n^{i}}{\lambda_n}\right)^2\\[1mm]
	&\geq\left(\mathbb{E}\left[\alpha^{-1}\right]\right)^{-1}\liminf_{n\rightarrow\infty}\hspace{-4mm}\sum_{\footnotesize\substack{m=0\\ (nI_m\cap\mathbb{Z} )\subset Q_n(\rho)}}^{M}\hspace{-4mm}\chi_{n,M}^{m}\lambda_n\left(i_{\mathrm{max}}^m-i_{\mathrm{min}}^m+1\right)\left(\dfrac{v_n^{i_{\mathrm{max}}^{m}+1}-v_n^{i_{\mathrm{min}}^m}}{\lambda_n\left(i_{\mathrm{max}}^m-i_{\mathrm{min}}^m+1\right)} \right)^2.
	\end{split}	
	\end{align*}
	For the next term, since by construction it holds true that $\lambda_n\left(i_{\mathrm{max}}^m-i_{\mathrm{min}}^m+1\right)\rightarrow\left(t_{m+1}-t_m\right)$, we can again follow the proof of the liminf-inequality of Theorem~3.1 in \cite{unserPaper1} and obtain pointwise convergence of
	\begin{align}
	\label{pointwiseconvergence}
	\dfrac{v_n^{i_{\mathrm{max}}^{m}+1}-v_n^{i_{\mathrm{min}}^m}}{\lambda_n\left(i_{\mathrm{max}}^m-i_{\mathrm{min}}^m+1\right)}\rightarrow\dfrac{v_M(t_{m+1})-v_M(t_m)}{t_{m+1}-t_m},
	\end{align}
	since it holds true that $i_{\mathrm{max}}^m+1=i_{\mathrm{min}}^{m+1}$ and $v_M$ is defined as the piecewise affine interpolation of $v$ with grid points $t_m$. For writing the sum as an integral, we define
	\begin{align*}
	S_{\rho,M,n}&:=\bigcup_{\footnotesize\substack{\footnotesize\substack{m=0\\ I_m\cap S_n(\rho)\neq\emptyset}}}^M I_m=\bigcup_{\footnotesize\substack{m=0\\ I_m\cap S_n(\rho)\neq\emptyset}}^M\left[t_m,t_{m+1}\right),\\[1mm]
	S_{\rho,M}&:=\bigcup_{\footnotesize\substack{m=0\\ I_m\cap S_{\rho}\neq\emptyset}}^M I_m=\bigcup_{\footnotesize\substack{m=0\\I_m\cap S_{\rho}\neq\emptyset}}^M\left[t_m,t_{m+1}\right),
	\end{align*}
	where, by definition, $S_{\rho,M,n}= S_{\rho,M}$ holds true for $n$ large enough. Therefore, with the definition of $v_{n,M}$ as the piecewise affine interpolation of $v_n$ with respect to $i_{\mathrm{min}}^m$, we get	
	\begin{align}
	\label{fastzielQ}
	\begin{split}
	\liminf_{n\rightarrow\infty}\hspace{-1mm}\sum_{i\in Q_n(\rho)}\hspace{-2mm}\lambda_n\chi_n^{i} \alpha(\tau_i\omega)\left(\dfrac{v_n^{i+1}-v_n^{i}}{\lambda_n}\right)^2&\geq\left(\mathbb{E}\left[\alpha^{-1}\right]\right)^{-1}\liminf_{n\rightarrow\infty}\int_{(0,1)\setminus S_{\rho,M,n}} \chi_{n,M}|v_{n,M}'|^2\,\mathrm{d}x\\[1mm]
	&\geq\left(\mathbb{E}[\alpha^{-1}]\right)^{-1}\int_{(0,1)\setminus S_{\rho,M}}\left|v_M'(x)\right|^2\,\mathrm{d}x,
	\end{split}
	\end{align}
	where the last inequality follows from the weak lower semicontinuity of the $L^2$ norm and because we have $\chi_{n,M}\rightarrow1$ bounded in measure, $v'_{n,M}\rightarrow v'_M$ in $L^2$ by \eqref{pointwiseconvergence} and thus $\chi_{n,M}v'_{n,M}\rightharpoonup v_M'$ in $L^2$.

	It remains to perform the limit $\liminf\limits_{M\rightarrow\infty}$. Since the left hand side of \eqref{fastzielQ} is independent of $M$ we only have to consider
	\begin{align}
	\label{limMzwei}
\liminf\limits_{M\rightarrow\infty}\int_{(0,1)\setminus S_{\rho,M}}\left|v_M'(x)\right|^2\,\mathrm{d}x
	&\geq\liminf\limits_{M\rightarrow\infty}\int_{(0,1)\setminus S_{\rho}}\left|v_M'(x)\right|^2\,\mathrm{d}x-\limsup\limits_{M\rightarrow\infty}\int_{S_{\rho,M}\setminus S_{\rho}}\left|v_M'(x)\right|^2\,\mathrm{d}x.
	\end{align}
	Step E of the liminf-inequality in the proof of  \cite[Theorem~3.1]{unserPaper1} shows that $v_M\rightharpoonup^* v$ in $BV$, and therefore $v_M\to v$ in $L^1(0,1)$. The compactness result in Theorem~\ref{Thm:compactnessgammarescaled} further yields $v'\in L^2(0,1)$. For an interval $(a,b)\subset(0,1)\setminus S_{\rho}$ we define the coarser grid points as before, with $t_0=a$ and $t_M=b$. Using the Hölder inequality, which is possible due to $v'=\nabla v$ on $(a,b)$, we get the uniform bound
	\begin{align*}
	\lVert v_M'\rVert^2_{L^2(a,b)}&=\int_{a}^{b}|v_M'(x)|^2\,\mathrm{d}x=\sum_{m=0}^{M}\left( t_{m+1}-t_m \right)\left|\frac{v(t_{m+1})-v(t_m)}{t_{m+1}-t_m} \right|^2\\[2mm]
	&=\sum_{m=0}^{M}\frac{1}{t_{m+1}-t_m}\left|\int_{t_{m}}^{t_{m+1}}v'(x)\,\mathrm{d}x \right|^2\leq \sum_{m=0}^{M}\int_{t_{m}}^{t_{m+1}}\left|v'(x)\right|^2\,\mathrm{d}x= \lVert v'\rVert^2_{L^2(a,b)},
	\end{align*}
	which yields $v_M'\to v'$ in $L^2(a,b)$ as $M\to\infty$. This result can be applied to $(0,1)\setminus S_\rho$ and reads $v_M'\to v'$ in $L^2((0,1)\setminus S_\rho)$ as $M\to\infty$. Since the integral functional is lower semicontinuous, we can estimate the first term of the right hand side of \eqref{limMzwei} by
	\begin{align*}
	\liminf\limits_{M\rightarrow\infty}\int_{(0,1)\setminus S_{\rho}}\left|v_M'(x)\right|^2\,\mathrm{d}x\geq\int_{(0,1)\setminus S_{\rho}}\left|v'(x) \right|^2\,\mathrm{d}x.
	\end{align*}
	The second part of \eqref{limMzwei} fulfils
	\begin{align*}
	\limsup\limits_{M\rightarrow\infty}\int_{S_{\rho,M}\setminus S_{\rho}}\left|v_M'(x)\right|^2\,\mathrm{d}x=0,
	\end{align*}
	which can be seen as follows: We assume that $S_{\rho}$ consists only of one interval (which corresponds to $S=\{x_1\}$) and note that the proof for finitely many intervals is analogous. By construction, there exist sequences $k_M$ and $\ell_M$ such that
%	\begin{align*}
	$S_{\rho,M}\setminus S_{\rho}\subset\left(\left[t_{k_M},t_{k_M+1} \right]\cup\left[t_{\ell_M},t_{\ell_M+1} \right] \right)$
	%\end{align*}
	and 
	\begin{align}
	\label{integralgrenzenkonv}
	\begin{split}
	t_{k_M}\rightarrow x_i-\rho,&\quad t_{k_M+1}\rightarrow x_i-\rho,\\[1mm]
	t_{\ell_M}\rightarrow x_i+\rho,&\quad t_{\ell_M+1}\rightarrow x_i+\rho\quad\text{as}\ M\rightarrow\infty.
	\end{split}
	\end{align}
	Therefore, it holds true that
	\begin{align*}
	\int_{S_{\rho,M}\setminus S_{\rho}}\left|v_M'(x)\right|^2\,\mathrm{d}x&\leq\int_{t_{k_M}}^{t_{k_M+1}}\left|v_M'(x)\right|^2\,\mathrm{d}x+\int_{t_{\ell_M}}^{t_{\ell_M+1}}\left|v_M'(x)\right|^2\,\mathrm{d}x.
	\end{align*}
	We are going to consider only one of these two terms, because they have basically the same structure, and show that it converges to zero. Observing that the integration area is contained in $(0,1)\setminus S$ for $M\gg\rho$ (in fact, for $2\delta<\rho$) and therefore $v$ can be assumed as the absolutely continuous representative, we get with the Hölder inequality
	\begin{align*}
	&\int_{t_{k_M}}^{t_{k_M+1}}\left|v_M'(x)\right|^2\,\mathrm{d}x%=\left(t_{k_M+1}-t_{k_M}\right)\left|\dfrac{v\left(t_{k_M+1}\right)-v\left(t_{k_M}\right)}{t_{k_M+1}-t_{k_M}}\right|^2=\dfrac{\left|v\left(t_{k_M+1}\right)-v\left(t_{k_M}\right)\right|^2}{t_{k_M+1}-t_{k_M}}\\[2mm]
	=\dfrac{1}{t_{k_M+1}-t_{k_M}}\left|\int_{t_{k_M}}^{t_{k_M+1}}v'(x)\,\mathrm{d}x\right|^2\leq \dfrac{1}{t_{k_M+1}-t_{k_M}}\left(\int_{t_{k_M}}^{t_{k_M+1}}\left|v'(x)\right|\,\mathrm{d}x\right)^2\\[2mm]
	&\leq\dfrac{1}{t_{k_M+1}-t_{k_M}}\left[\left(\int_{t_{k_M}}^{t_{k_M+1}}\left|v'(x)\right|^2\,\mathrm{d}x\right)^{\frac{1}{2}}\left(t_{k_M+1}-t_{k_M}\right)^{\frac{1}{2}} \right]^2=\int_{t_{k_M}}^{t_{k_M+1}}\left|v'(x)\right|^2\,\mathrm{d}x.
	\end{align*}
	Since $v$ is absolutely continuous, the integral functional is continuous with respect to its integral bounds due to the fundamental theorem of calculus. Together with \eqref{integralgrenzenkonv}, this shows
	\begin{align*}
	\limsup_{M\rightarrow\infty}\int_{t_{k_M}}^{t_{k_M+1}}\left|v'(x)\right|^2\,\mathrm{d}x=\lim\limits_{M\rightarrow\infty}\int_{t_{k_M}}^{t_{k_M+1}}\left|v'(x)\right|^2\,\mathrm{d}x=0.
	\end{align*}
	Summarizing steps A--C, we obtain \eqref{zielliminfQ}.
	
	\medskip
	
	\noindent
	\textit{Step D: Proof of the jump part of the energy \eqref{zielliminfS}.} 
	
	It is left to show
	\begin{align*}
	\liminf_{n\rightarrow\infty}\sum_{i\in S_n(\rho)}\left(J\left(\tau_i\omega,\dfrac{v_n^{i+1}-v_n^{i}}{\sqrt{\lambda_n}}+\delta(\tau_i\omega) \right)-J\left(\tau_i\omega,\delta(\tau_i\omega) \right) \right)\geq\beta\#S_v.
	\end{align*}
	According to \cite[(117)]{BraidesLewOrtiz2006}, one can find a sequence $(h_n^t)\subset\mathbb{N}$ for every $t\in S_v$ with $\lambda_n h_n^t\rightarrow t$ as $n\rightarrow\infty$ such that
	\begin{align}
	\label{sprungfolgekonv}
	\lim\limits_{n\rightarrow \infty}\dfrac{v_n^{h_n^t+1}-v_n^{h_n^t}}{\sqrt{\lambda_n}}=+\infty.
	\end{align}
	Especially, $h_n^t\notin Q_n(\rho)$ holds true for $n$ large enough.
	The existence of such a sequence can be seen by a contradiction argument: If this did not exist, we would get $v_n'<C/\sqrt{\lambda_n}$ in a neighbourhood $(t-\xi,t+\xi)$ of $t$. Following step 1 of the proof of Theorem~\ref{Thm:compactnessgammarescaled}, this would imply $\int_{t-\xi}^{t+\xi}|v_n'|^2\,\mathrm{d}t\leq CE_n^{\gamma_n}(\omega,v_n)$, and therefore $v_n$ would be equibounded in $H^1(0,1)$ in a neighbourhood of $t$. Therefore, we get
	\begin{align*}
	&\sum_{i\in S_n(\rho)}\left(J\left(\tau_i\omega,\dfrac{v_n^{i+1}-v_n^{i}}{\sqrt{\lambda_n}}+\delta(\tau_i\omega) \right)-J\left(\tau_i\omega,\delta(\tau_i\omega) \right) \right)\\[1mm]
	&\geq\sum_{t\in S_v}\left(J\left(\tau_{h_n^t}\omega,\dfrac{v_n^{h_n^t+1}-v_n^{h_n^t}}{\sqrt{\lambda_n}}+\delta\left(\tau_{h_n^t}\omega\right) \right)-J\left(\tau_{h_n^t}\omega,\delta\left(\tau_{h_n^t}\omega\right) \right) \right)\\[1mm]
	&\geq\sum_{t\in S_v}J\left(\tau_{h_n^t}\omega,\dfrac{v_n^{h_n^t+1}-v_n^{h_n^t}}{\sqrt{\lambda_n}}+\delta\left(\tau_{h_n^t}\omega\right) \right)+\sum_{t\in S_v}\inf\{-J\left(\omega,\delta(\omega) \right)\, :\,\omega\in\Omega \}\\[1mm]
	&=\sum_{t\in S_v}J\left(\tau_{h_n^t}\omega,\dfrac{v_n^{h_n^t+1}-v_n^{h_n^t}}{\sqrt{\lambda_n}}+\delta\left(\tau_{h_n^t}\omega\right) \right)+\beta\#S_v.
	\end{align*}
	By taking $\liminf_{n\rightarrow\infty}$, the first term vanishes because of \eqref{sprungfolgekonv} and (R2). % which says that every Lennard-Jones type potential uniformly converges to $0$ as $z\rightarrow\infty$.\\
	
	\medskip
	
	\noindent
	\textbf{Step 2.} Limsup inequality.
	
	We have to show that for every $v\in SBV_c^{\gamma}$ there exists a sequence $v_n$ with $v_n\rightarrow v$ in $L^1(0,1)$ such that
%	\begin{align*}
	%\label{ziellimsup}
	$\limsup_{n\rightarrow\infty}E_n^{\gamma_n}(\omega,v_n)\leq E^{\gamma}(v)$.
%	\end{align*}
	Without loss of generality, we consider $\# S_v=1$ to keep the notation simple. The extension to the case $\# S_v>1$ can easily be proven since we construct the recovery sequence step by step, starting from affine functions and glueing them together to a piecewise affine function. The case $\# S_v=0$ is included by setting the jump height to zero. % which simplifies the subsequent calculations.
	
	We already know from the compactness result in Theorem~\ref{Thm:compactnessgammarescaled} that $v$ is piecewise $H^1(0,1)$. Therefore, we can write $v=v_c+v_j$, where $v_c\in H^1(0,1)$ and $v_j$ is a piecewise constant function. By a density argument, we can assume $v_c\in C^{2}[0,1]$. Note that this approximation can be chosen in such a way, that it keeps the boundary values, see e.g.~\cite[Section 2.4, Cor.~3]{Burenkov}. We first provide a recovery sequence for affine functions with a single jump and extend it afterwards to the general case.
	
	\medskip
	
	\noindent
	\textit{Step A: Affine function.}
	
	We construct a recovery sequence for an affine function $v$ with slope $z$, a jump in $0$ and constant near the jump. That is, for $z\in \mathrm{dom} J_j$ and a small $\rho>0$, we have
	\begin{align*}
	v(x)=\begin{cases}
	0&\quad\text{for}\ x=0,\\
	v(0^+)&\quad\text{for}\ x\in(0,\rho),\\
	v(0^+)+(x-\rho)z&\quad\text{for}\ x\in[\rho,1],
	\end{cases}
	\end{align*}
	with $v(0^+)>0$ defined such that
	\begin{align}
	\label{boundaryvaluev}
	v(1)=v(0^+)+(1-\rho)z=\gamma,
	\end{align}
	in order to fulfil the boundary constraint as well as the assumption $\# S_v=1$ and $[v]\geq 0$ from the compactness result in Theorem~\ref{Thm:compactnessgammarescaled}. Without loss of generality, the jump is at $0$. A jump at $1$ can be constructed analogously, with small changes. Whenever the location of the jump is important in the proof (cf.\ \eqref{limesbeta}), we highlight it and provide the proof in a general way.
	
	There exists a unique sequence $T_n$ such that $\rho\in[T_n,T_n+1)\lambda_n$. First, we consider $\epsilon>0$ fixed such that $\epsilon<\rho$ and $\epsilon<T_n\lambda_n$ and define
	\begin{align} \label{hnepsilon}
	h_n^{\epsilon}:=\mathrm{argmin}_{0\leq i\leq n-1}\left\{-J(\tau_i\omega,\delta(\tau_i\omega))\, :\,\left|i\lambda_n-0 \right|<\epsilon \right\}.
	\end{align}
	Let $(\mu_k)_{k\in\mathbb{N}}>0$ be a sequence in $\mathbb{R}$ with $\mu_k=\frac{1-\rho}{k}$ for all $k\in\mathbb{N}$. We define a partition of the interval $(\rho,1]$ by $I_{j}^{\mu_k}:=(\rho+j\mu_k,\rho+(j+1)\mu_k]$ with $j=0,\ldots,\frac{1-\rho}{\mu_k}-1$. Hence, the set
	\begin{align*}
	\left\{I_j^{\mu_k} \, :\, \mu_k=\frac{1-\rho}{k},\ j=0,\ldots,k-1,\ k\in\mathbb{N} \right\}
	\end{align*}
	is a countable set of sets. Thus, for all $I=I_j^{\mu_k}$ with $\mu_k=\frac{1}{k}$, $k\in\mathbb{N}$ and $j=0,\ldots,k-1$, we can pass to the limit in the sense of Tempel'man's ergodic theorem \cite[Theorem~2.8]{Krengel}. Thereby, the set $\Omega'$, for which the ergodic result holds true, is the intersection of countably many sets $\Omega_I$. In the following, we leave out the index $k$ and just refer to the sequence $\mu_k$ by $\mu$.
	From now on, let $\mu$ be fixed. We define for $j=0,\ldots,(1-\rho)/\mu-1=:j_{\mathrm{max}}$
	\begin{align*}
	I_{j}^{\mu}&:=(\rho+j\mu,\rho+(j+1)\mu],\\[1mm]
	I_{j,n}^{\mu}&:=\mathbb{Z}\cap nI_{j}^{\mu},\\[1mm]
	i_{\mathrm{min}}^{j,n}&:=\min\{i\, :\,i\in \mathbb{Z}\cap n I_{j}^{\mu}\},\\[1mm]
	i_{\mathrm{max}}^{j,n}&:=\max\{i\, :\,i\in \mathbb{Z}\cap n I_{j}^{\mu}\},\\[1mm]
	I_{j,n}^{\mu*}&:=I_{j,n}^{\mu}\cup\{i_{\mathrm{min}}^{j,n}-1 \}\setminus\{i_{\mathrm{max}}^{j,n} \}.
	\end{align*}
	Note that $\cup_{j=0}^{j_{\mathrm{max}}}I_{j,n}^{\mu*}=\{T_n,\ldots,n-1\}$ holds true.
	With this notation, we define two sequences of piecewise affine functions, $(\varphi_n)$ and $(\phi_n)$, which together (almost) form the recovery sequence $(v_n)$, defined by
	\begin{align*}
	v_n:=\varphi_n+\phi_n.	
	\end{align*}
	
	To be precise, the sequence $v_n$ has to be a sequence $v_{n,\mu}$, depending on $\mu$. Since this does not affect most of the calculations, we drop the subscript $\mu$ whenever it is not relevant. In the last step of the proof, we get from the Attouch Lemma \cite[Corollary~1.16]{Attouch1984} the existence of a sequence $\mu_n$, such that $(v_{n,\mu_n})$ finally is the desired recovery sequence. We first define $(\varphi_n)$, which accounts for the jump and the boundary conditions, by
	\begin{align*}
	\varphi_n^i:=\begin{cases}
	v(0^-)=0&\ \text{for}\ 0\leq i\leq h_n^{\epsilon},\\
	v(0^+)+\gamma_n-\gamma&\ \text{for}\ h_n^{\epsilon}<i\leq n.
	\end{cases}
	\end{align*}
	
	Since later we extend this construction to piecewise affine functions, we have to take care of the boundary values $\gamma_n$ and $\gamma$. They have to be understood as the boundary data of the considered interval. That is, for an interval which does not contain $x=0$ or $x=1$, it holds true that $\gamma_n=\gamma$ and therefore the term $\gamma_n-\gamma$ cancels out.
	
	The sequence $(\phi_n)$ optimizes the elastic energy and is given by
	\begin{align*}
	\phi_n^i:=\begin{cases}
	0&\ \text{for}\quad 0\leq i\leq h_n^{\epsilon},\\
	z\left(T_n\lambda_n-\rho \right)&\ \text{for}\quad h_n^\epsilon<i\leq T_n,\\[-1mm]
	z\left((i_{\mathrm{min}}^{j,n}-1)\lambda_n-\rho \right)+\left(\dfrac{1}{|I_{j,n}^{\mu*}|}\displaystyle\sum_{k\in I_{j,n}^{\mu*}}\dfrac{1}{\alpha(\tau_k\omega)} \right)^{-1}\hspace{-3mm}\lambda_nz\hspace{-2mm}\displaystyle\sum_{k=i_{\mathrm{min}}^{j,n}-1}^{i-1}\dfrac{1}{\alpha(\tau_k\omega)}&\ \text{for}\quad i\in I_{j,n}^{\mu}.
	\end{cases}
	\end{align*}
	Note that the definition provides for every $j=0,\ldots,(1-\rho)/\mu-1$
	\begin{align}
	\label{boundaryvalueinterval}
	\begin{split}
	v_n^{i_{\mathrm{max}}^{j,n}}&=\varphi_n^{i_{\mathrm{max}}^{j,n}}+\phi_n^{i_{\mathrm{max}}^{j,n}}\\[-3mm]
	&=v(0^+)+\gamma_n-\gamma+z\left((i_{\mathrm{min}}^{j,n}-1)\lambda_n-\rho \right)+\left(\dfrac{1}{|I_{j,n}^{\mu*}|}\displaystyle\sum_{k\in I_{j,n}^{\mu*}}\dfrac{1}{\alpha(\tau_k\omega)} \right)^{-1}\hspace{-3mm}\lambda_nz\hspace{-2mm}\sum_{k=i_{\mathrm{min}}^{j,n}-1}^{i_{\mathrm{max}}^{j,n}-1}\dfrac{1}{\alpha(\tau_k\omega)}\\[1mm]
	&=v(0^+)+\gamma_n-\gamma+z\left((i_{\mathrm{min}}^{j,n}-1)\lambda_n-\rho \right)+\lambda_nz|I_{j,n}^{\mu*}|\\[1mm]
	&=v(0^+)+\gamma_n-\gamma+z\left((i_{\mathrm{min}}^{j,n}-1)\lambda_n-\rho \right)+\lambda_nz\left(i_{\mathrm{max}}^{j,n}-i_{\mathrm{min}}^{j,n}+1  \right)\\[1mm]
	&=v(0^+)+\gamma_n-\gamma+z\left((i_{\mathrm{max}}^{j,n})\lambda_n-\rho \right) =v^{i_{\mathrm{max}}^{j,n}}+\gamma_n-\gamma.
	\end{split}
	\end{align}
	Therefore, $v_n$ and $v$ coincide at the value $i_{\mathrm{max}}^{j,n}$ up to their boundary conditions. Together, the sequence $(v_{n,\mu_n})$ of piecewise affine functions $v_n\in \mathcal{A}_n^{\gamma_n}(0,1)$,  with $v_n(0)=0$, $v_n(1)=\gamma_n$ and $v_n:=\varphi_n+\phi_n$, is the recovery sequence for a well-chosen $\mu_n$. To prove this, we have to show that (a) $v_n$ fulfils the boundary conditions, (b) the limsup inequality is satisfied and (c) $v_n\rightarrow v$ in $L^1(0,1)$.
	
	\medskip 
	
	\noindent
	(a) We consider the point $i=n$ because $v(1)=v_n^n$. Since $n=i_{\mathrm{max}}^{j_{\mathrm{max}},n}$, by \eqref{boundaryvalueinterval} we get $v_n^n=v_n^{i_{\mathrm{max}}^{j_{\mathrm{max}},n}}=v(0^+)+z(n\lambda_n-\rho)+\gamma_n-\gamma\stackrel{\eqref{boundaryvaluev}}{=}\gamma_n$. Thus, the boundary condition is fulfilled.\\
	
	\noindent
	(b) We have
	\begin{align*}
	E_n^{\gamma_n}(\omega,v_n)&=\sum_{i=0}^{n-1}\left(J\left(\tau_i\omega,\dfrac{v_n^{i+1}-v_n^{i}}{\sqrt{\lambda_n}}+\delta(\tau_i\omega)\right)-J\left(\tau_i\omega,\delta(\tau_i\omega) \right)\right)\\[1mm]
	&=\sum_{i=0}^{h_n^{\epsilon}-1}\left(J\left(\tau_i\omega,\dfrac{v_n^{i+1}-v_n^{i}}{\sqrt{\lambda_n}}+\delta(\tau_i\omega)\right)-J\left(\tau_i\omega,\delta(\tau_i\omega) \right)\right)\\[1mm]
	&\quad\quad+\sum_{i=h_n^{\epsilon}+1}^{T_n-1}\left(J\left(\tau_i\omega,\dfrac{v_n^{i+1}-v_n^{i}}{\sqrt{\lambda_n}}+\delta(\tau_i\omega)\right)-J\left(\tau_i\omega,\delta(\tau_i\omega) \right)\right)\\[1mm]
	&\quad\quad+\sum_{i=T_n}^{n-1}\left(\lambda_n\alpha(\tau_i\omega)\left(\dfrac{v_n^{i+1}-v_n^{i}}{\lambda_n} \right)^2+\eta\left(\tau_i\omega,\dfrac{v_n^{i+1}-v_n^{i}}{\sqrt{\lambda_n}}\right)\right)\\[1mm]
	&\quad\quad+J\left(\tau_{h_n^{\epsilon}}\omega,\dfrac{v_n^{h_n^{\epsilon}+1}-v_n^{h_n^{\epsilon}}}{\sqrt{\lambda_n}}+\delta\left(\tau_{h_n^{\epsilon}}\omega\right)\right)-J(\tau_{h_n^{\epsilon}}\omega,\delta\left(\tau_{h_n^{\epsilon}}\omega\right)).
	\end{align*}	
	This energy has four parts. The first two parts (from zero to $h_n^{\epsilon}-1$ and from $h_n^{\epsilon}+1$ to $T_n-1$) are identically zero by definition of $v_n$. To get the limsup-inequality, we have to show the two inequalities
	\begin{align}
	\label{limsup1}
	\limsup_{n\to\infty}\sum_{i=T_n}^{n-1}\left(\lambda_n\alpha(\tau_i\omega)\left(\dfrac{v_n^{i+1}-v_n^{i}}{\lambda_n} \right)^2+\eta\left(\tau_i\omega,\dfrac{v_n^{i+1}-v_n^{i}}{\sqrt{\lambda_n}}\right)\right)\leq \underline{\alpha}\int_{0}^{1}|v'(x)|^2,
	\end{align}
	and
	\begin{align}
	\label{limsup2}
	\limsup_{n\to\infty}\left( J\left(\tau_{h_n^{\epsilon}}\omega,\dfrac{v_n^{h_n^{\epsilon}+1}-v_n^{h_n^{\epsilon}}}{\sqrt{\lambda_n}}+\delta\left(\tau_{h_n^{\epsilon}}\omega\right)\right)-J(\tau_{h_n^{\epsilon}}\omega,\delta\left(\tau_{h_n^{\epsilon}}\omega\right))\right)\leq \beta,
	\end{align}
	where the first one is the elastic part and the second one is the jump part of the limiting energy.\\
	
	\noindent
	\textit{Proof of equation \eqref{limsup1}, elastic part.} 
	
	We start with rearranging the sum, i.e.
	\begin{align*}
	&\sum_{i=T_n}^{n-1}\left(\lambda_n\alpha(\tau_i\omega)\left(\dfrac{v_n^{i+1}-v_n^{i}}{\lambda_n} \right)^2+\eta\left(\tau_i\omega,\dfrac{v_n^{i+1}-v_n^{i}}{\sqrt{\lambda_n}}\right)\right)\\[1mm]
	&=\sum_{j=0}^{j_{\mathrm{max}}}\sum_{i\in I_{j,n}^{\mu*}}\lambda_n\alpha(\tau_i\omega)\left(\dfrac{v_n^{i+1}-v_n^{i}}{\lambda_n} \right)^2+\sum_{j=0}^{j_{\mathrm{max}}}\sum_{i\in I_{j,n}^{\mu*}}\eta\left(\tau_i\omega,\dfrac{v_n^{i+1}-v_n^{i}}{\sqrt{\lambda_n}}\right).
	\end{align*}
	By the definition of $v_n$, the first term on the right hand side reads
	\begin{align*}
%	\sum_{i\in I_{j,n}^{\mu*}}\lambda_n\alpha(\tau_i\omega)\left(\dfrac{v_n^{i+1}-v_n^{i}}{\lambda_n} \right)^2&=
\sum_{j=0}^{j_{\mathrm{max}}}\sum_{i\in I_{j,n}^{\mu*}}\lambda_n\dfrac{1}{\alpha(\tau_i\omega)}\left(\dfrac{1}{|I_{j,n}^{\mu*}|}\sum_{k\in I_{j,n}^{\mu*}}\dfrac{1}{\alpha(\tau_k\omega)} \right)^{-2}z^2 %\\[1mm]
	%&
	=\sum_{j=0}^{j_{\mathrm{max}}}\lambda_n|I_{j,n}^{\mu*}|z^2\left(\dfrac{1}{|I_{j,n}^{\mu*}|}\sum_{k\in I_{j,n}^{\mu*}}\dfrac{1}{\alpha(\tau_k\omega)} \right)^{-1}.
	\end{align*}
Hence we obtain
	\begin{align}
	\label{energyiandii}
	\begin{split}
	&\sum_{i=T_n}^{n-1}\left(\lambda_n\alpha(\tau_i\omega)\left(\dfrac{v_n^{i+1}-v_n^{i}}{\lambda_n} \right)^2+\eta\left(\tau_i\omega,\dfrac{v_n^{i+1}-v_n^{i}}{\sqrt{\lambda_n}}\right)\right)\\[1mm]
	&=\sum_{j=0}^{j_{\mathrm{max}}}\lambda_n|I_{j,n}^{\mu*}|z^2\left(\dfrac{1}{|I_{j,n}^{\mu*}|}\sum_{k\in I_{j,n}^{\mu*}}\dfrac{1}{\alpha(\tau_k\omega)} \right)^{-1} \\[1mm]
	&\quad\quad+\sum_{j=0}^{j_{\mathrm{max}}}\sum_{i\in I_{j,n}^{\mu*}}\eta\left(\tau_i\omega,\sqrt{\lambda_n}\left(\dfrac{1}{|I_{j,n}^{\mu*}|}\sum_{k\in I_{j,n}^{\mu*}}\dfrac{1}{\alpha(\tau_k\omega)} \right)^{-1}z\dfrac{1}{\alpha(\tau_i\omega)}\right).
	\end{split}
	\end{align}
	Now we consider $\limsup_{n\rightarrow\infty}$ of \eqref{energyiandii}. The two parts of the sum are discussed separately in (i) and (ii) below. The first one becomes the elastic part of the energy and the second one vanishes.

	%	i) The first part can be estimated by
	%	\begin{align*}
	%	0&\leq\lambda_n\alpha_{\tau_{T_n-1}\omega}\left(\dfrac{z\left(T_n\lambda_n-\rho \right)}{\lambda_n}\right)^2\leq \lambda_nz^2\alpha_{\tau_{T_n-1}\omega}\left|\dfrac{\lambda_n }{\lambda_n}\right|^2\\
	%	&=\lambda_n\alpha_{\tau_{T_n-1}\omega}z^2\leq\lambda_n\alpha_{max}z^2\rightarrow0
	%	\end{align*}
	%	for $n\to\infty$.\footnote{alpha max}
	%	
	%	ii) Since $\rho-T_n\lambda_n\leq\lambda_n$, the second part can be estimated by
	%	\begin{align*}
	%	\left|\eta_{\tau_{T_n}\omega}\left(\dfrac{z\left(T_n\lambda_n-\rho\right)}{\sqrt{\lambda_n}} \right)\right|\leq\sup_{|t|\leq \frac{z\left(\rho-T_n\lambda_n\right)}{\sqrt{\lambda_n}}}\left|\eta_{\tau_{T_n}\omega}\left(t \right)\right|\leq\sup_{|t|\leq z\sqrt{\lambda_n}}\left|\eta_{\tau_{T_n}\omega}\left(t \right)\right|.
	%	\end{align*}
	%	The Taylor expansion \eqref{etadurchx} yields $\eta_{\tau_{T_n}\omega}(x)\to 0$ for $|x|\to 0$, and therefore FEHLT \footnote{konvergiert nicht gleichmäßig...}
	
	\noindent
	(i) %The first part of \eqref{energyiandii} is 
% 	\begin{align*}
% 	\sum_{j=0}^{j_{\mathrm{max}}}\lambda_n|I_{j,n}^{\mu*}|z^2\left(\dfrac{1}{|I_{j,n}^{\mu*}|}\sum_{k\in I_{j,n}^{\mu*}}\dfrac{1}{\alpha(\tau_k\omega)} \right)^{-1}.
% 	\end{align*}
	%	Since we are interested in the limsup-inequality, we can estimate upwards as
	%	\begin{align*}
	%		\sum_{j=0}^{j_{\mathrm{max}}}\lambda_n|I_{j,n}^{\eta*}|z^2\left(\dfrac{1}{|I_{j,n}^{\eta*}|}\sum_{k\in I_{j,n}^{\eta*}}\dfrac{1}{\alpha(\tau_k\omega)} \right)^{-1}&\leq \sum_{j=0}^{j_{\mathrm{max}}}\eta z^2\left(\dfrac{\lambda_n}{\eta}\sum_{k\in I_{j,n}^{\eta*}}\dfrac{1}{\alpha(\tau_k\omega)} \right)^{-1}\\
	%		&= \sum_{j=0}^{j_{\mathrm{max}}}\eta^2 z^2\left(\left(\lambda_n\sum_{k\in \mathbb{Z}\cap n I_j^{\eta}}\dfrac{1}{\alpha(\tau_k\omega)} \right)-\dfrac{\lambda_n}{\alpha_{\tau_{i_{\mathrm{max}}^{j,n}\omega}}}\right)^{-1}.
	%	\end{align*}
	We take the $\limsup\limits_{n\rightarrow\infty}$ of the first term of the right hand side of this equation and obtain with Proposition~\ref{Prop:averages2} and $\lambda_n|I_{j,n}^{\mu*}|\to\mu$ that
	%	\begin{align*}
	%		&\lim\limits_{n\rightarrow\infty}\sum_{j=0}^{j_{\mathrm{max}}}\eta^2 z^2\left(\left(\lambda_n\sum_{k\in \mathbb{Z}\cap n I_j^{\eta}}\dfrac{1}{\alpha(\tau_k\omega)}\right) -\dfrac{\lambda_n}{\alpha_{\tau_{i_{\mathrm{max}}^{j,n}\omega}}}\right)^{-1}=\sum_{j=0}^{j_{\mathrm{max}}}\eta^2 z^2\left(\eta\cdot\mathbb{E}[\frac{1}{\alpha}] \right)^{-1}\\
	%		&=\dfrac{1-\rho}{\eta}\eta^2 z^2\left(\eta\cdot\mathbb{E}\left[\frac{1}{\alpha}\right] \right)^{-1}=(1-\rho)z^2\left(\mathbb{E}\left[\frac{1}{\alpha}\right] \right)^{-1}=\underline{\alpha}\int_{\rho}^{1}z^2\,\mathrm{d}x=\underline{\alpha}\int_{0}^{1}|v'(x)|^2\,\mathrm{d}x
	%	\end{align*}
	\begin{align*}
	&\limsup_{n\to\infty}\sum_{j=0}^{j_{\mathrm{max}}}\lambda_n|I_{j,n}^{\mu*}|z^2\left(\dfrac{1}{|I_{j,n}^{\mu*}|}\sum_{k\in I_{j,n}^{\mu*}}\dfrac{1}{\alpha(\tau_k\omega)} \right)^{-1}\\[1mm]
	&\leq \sum_{j=0}^{j_{\mathrm{max}}}\left(\limsup_{n\to\infty}\lambda_n|I_{j,n}^{\mu*}|\right)z^2\limsup_{n\to\infty}\left(\dfrac{1}{|I_{j,n}^{\mu*}|}\sum_{k\in I_{j,n}^{\mu*}}\dfrac{1}{\alpha(\tau_k\omega)} \right)^{-1}\\[2mm]
	&=\dfrac{1-\rho}{\mu}\mu z^2\left(\mathbb{E}\left[\alpha^{-1}\right]\right)^{-1}=(1-\rho)z^2\left(\mathbb{E}\left[\alpha^{-1}\right]\right)^{-1}=\underline{\alpha}\int_{\rho}^{1}z^2\,\mathrm{d}x=\underline{\alpha}\int_{0}^{1}|v'(x)|^2\,\mathrm{d}x,
	\end{align*}
	$\mathbb{P}$-almost everywhere. This is exactly the result we expected in order to get \eqref{limsup1}. We now show that the remaining part of \eqref{energyiandii} vanishes, which will conclude the proof of \eqref{limsup1}.
	
	\medskip 
	
	\noindent
	(ii) To estimate the second part on the right hand side of \eqref{energyiandii}, we
we first consider the second argument of the function $\eta$. Since $\alpha(\omega)$ is bounded from below due to Remark~\ref{Rm:alphaneu} (iii), we get
	\begin{align*}
	\sqrt{\lambda_n}\left(\dfrac{1}{|I_{j,n}^{\mu*}|}\sum_{i\in I_{j,n}^{\mu*}}\dfrac{1}{\alpha(\tau_k\omega)} \right)^{-1}|z|\dfrac{1}{\alpha(\tau_i\omega)}\leq\sqrt{\lambda_n}C,
	\end{align*}
	because of the convergence of the sum to $\mathbb{E}\left[\alpha^{-1}\right]$, due to Proposition~\ref{Prop:averages2}. As before, we use the Lagrange form of the remainder from \eqref{lagrangeremainder} and get with $\xi_i\in[\delta(\tau_i\omega)-\sqrt{\lambda_n}C,\delta(\tau_i\omega)+\sqrt{\lambda_n}C]$ 
	\begin{align*}
	&\sum_{j=0}^{j_{\mathrm{max}}}\sum_{i\in I_{j,n}^{\mu*}}\eta\left(\tau_i\omega,\sqrt{\lambda_n}\left(\dfrac{1}{|I_{j,n}^{\mu*}|}\sum_{i\in I_{j,n}^{\mu*}}\dfrac{1}{\alpha(\tau_k\omega)} \right)^{-1}z\dfrac{1}{\alpha(\tau_i\omega)}\right)\\[1mm]
	&=\sum_{j=0}^{j_{\mathrm{max}}}\sum_{i\in I_{j,n}^{\mu*}}\frac{1}{6}\left.\frac{\partial^3J(\tau_i\omega,y)}{\partial y^3}\right|_{y=\xi_i}   \left(\sqrt{\lambda_n}\left(\dfrac{1}{|I_{j,n}^{\mu*}|}\sum_{i\in I_{j,n}^{\mu*}}\dfrac{1}{\alpha(\tau_k\omega)} \right)^{-1}z\dfrac{1}{\alpha(\tau_i\omega)}\right)^3.
	\end{align*}
	We can again use the estimate from above and get with $\kappa<\kappa^*$ from (R1) for $n$ large enough
	\begin{align*}
	&\sum_{j=0}^{j_{\mathrm{max}}}\sum_{i\in I_{j,n}^{\mu*}}\frac{1}{6}\left|\left.\frac{\partial^3J(\tau_i\omega,y)}{\partial y^3}\right|_{y=\xi_i}\right| \left(\sqrt{\lambda_n}\left(\dfrac{1}{|I_{j,n}^{\mu*}|}\sum_{i\in I_{j,n}^{\mu*}}\dfrac{1}{\alpha(\tau_k\omega)} \right)^{-1}|z|\dfrac{1}{\alpha(\tau_i\omega)}\right)^3\\[1mm]
	&\leq\sum_{i=0}^{n-1}\frac{1}{6}\left|\left.\frac{\partial^3J(\tau_i\omega,y)}{\partial y^3}\right|_{y=\xi_i}\right| \left(\sqrt{\lambda_n}C\right)^3\\[1mm]
	&\leq\frac{1}{6}C^3\lambda_n^{\frac{1}{2}}\lambda_n\sum_{i=0}^{n-1}\sup_{x\in[\delta(\tau_i\omega)-\kappa,\delta(\tau_i\omega)+\kappa]}\left|\left.\frac{\partial^3J(\tau_i\omega,y)}{\partial y^3}\right|_{y=x}\right|\leq \hat{C}\lambda_n^{\frac{1}{2}},
	\end{align*}
	where the last estimate is due to the convergence of the random variable $C^\kappa$ to its expectation value, see Proposition~\ref{Prop:averages2}. Therefore, the whole expression converges to zero, which concludes the proof of Equation \eqref{limsup1}.\\

	\noindent
	\textit{Proof of equation \eqref{limsup2}, jump part.} 
	
	The last remaining part of the energy is the limsup of
	\begin{align*}
	J\left(\tau_{h_n^{\epsilon}}\omega,\dfrac{v_n^{h_n^{\epsilon}+1}-v_n^{h_n^{\epsilon}}}{\sqrt{\lambda_n}}+\delta\left(\tau_{h_n^{\epsilon}}\omega\right)\right)-J(\tau_{h_n^{\epsilon}}\omega,\delta\left(\tau_{h_n^{\epsilon}}\omega\right)).
	\end{align*}
	We have
	\begin{align*}
	\dfrac{v_n^{h_n^{\epsilon}+1}-v_n^{h_n^{\epsilon}}}{\sqrt{\lambda_n}}=\dfrac{v(0^+)+\gamma_n-\gamma+z\left(T_n\lambda_n-\rho \right)}{\sqrt{\lambda_n}}\rightarrow\infty
	\end{align*}
	as $n\rightarrow\infty$ since $\gamma_n-\gamma\rightarrow0$, $T_n\lambda_n\to \rho$ and $v(0^+)> 0$. Therefore, we obtain
	\begin{align*}
	J\left(\tau_{h_n^{\epsilon}}\omega,\dfrac{v_n^{h_n^{\epsilon}+1}-v_n^{h_n^{\epsilon}}}{\sqrt{\lambda_n}}+\delta\left(\tau_{h_n^{\epsilon}}\omega\right)\right)\rightarrow0
	\end{align*}
	due to (R2). By definition of $h_n^{\epsilon}$ it holds true that
	\begin{align*} %\label{def:betaomegaxepsilon}
	-J(\tau_{h_n^{\epsilon}}\omega,\delta\left(\tau_{h_n^{\epsilon}}\omega\right))=\inf_{0\leq i\leq n-1}\left\{-J(\tau_{i}\omega,\delta(\tau_{i}\omega))\, :\, |i\lambda_n-0|< \epsilon \right\} =: \beta_n(\omega,x,\epsilon).
	\end{align*}
% 	Since we need it in further work, we define
% 	\begin{align}
% 	\label{def:betaomegaxepsilon}
% 	\beta_n(\omega,x,\epsilon):=\inf_{0\leq i\leq n-1}\left\{-J(\tau_{i}\omega,\delta(\tau_{i}\omega))\, :\, |i\lambda_n-x|< \epsilon \right\}.
% 	\end{align}
	In our case, the jump is at $x=0$. Since we will extend this construction to piecewise affine functions, the jump needs also to be allowed to be located at any point in the interval $[0,1]$. Therefore, we have to show that the results also hold for an arbitrary $x$.

	For the result of equation \eqref{limsup2}, it is now left to show that for every $\omega\in\Omega'$ and every $x\in[0,1]$ and every $\epsilon>0$ it holds true that
	\begin{align}\label{limitbetan}
	\lim\limits_{n\rightarrow\infty}\beta_n(\omega,x,\epsilon)=\lim\limits_{n\rightarrow\infty}\inf_{0\leq i\leq n-1}\left\{-J(\tau_{i}\omega,\delta(\tau_{i}\omega))\, :\, |i\lambda_n-0|< \epsilon \right\}=\beta.
	\end{align}
	Since $\beta_n(\omega,x,\epsilon)\geq\beta$ holds true for every $\omega$, $n$, $x$ and $\epsilon$ by definition, we only need to prove 
	\begin{align}
	\label{limesbeta}
	\lim\limits_{n\rightarrow\infty}\beta_n(\omega,x,\epsilon)=\lim\limits_{n\rightarrow\infty}\inf_{0\leq i\leq n-1}\left\{-J(\tau_{i}\omega,\delta(\tau_{i}\omega))\, :\, |i\lambda_n-0|< \epsilon \right\}\leq\beta.
	\end{align}
	First, notice that $\beta_n(\omega,x,\epsilon)$ is bounded, because of the boundedness of $J$ by $\Psi$ due to (LJ2). Let $\beta_n(\omega,x,\epsilon)$ be an arbitrary subsequence (not relabelled). Then, there exists a further subsequence (again not relabelled) which is convergent due to Bolzano-Weierstraß. To conclude, we show that every subsequence of that type converges to the same limit independent of $\omega$ and $x$.
	%, we get convergence of the whole sequence, since then every subsequence has a further subsequence with the same limit. This is what we are going to prove in the following. 
	It holds true, with $I_x^{\epsilon}:=]x-\epsilon,x+\epsilon[$ and $k\in\mathbb{R}$, that
	\begin{align}
	\label{betaherleitung1}
	\begin{split}
	&\inf_{0\leq i\leq n-1}\left\{-J(\tau_{i}\omega,\delta(\tau_{i}\omega))\, :\, |i\lambda_n-x|< \epsilon \right\}\cdot\dfrac{1}{2\epsilon n}\sum_{i\in\mathbb{Z}\cap n I_x^{\epsilon}}\chi_{(-J(\tau_{i}\omega,\delta(\tau_{i}\omega))\leq k)}\\[1mm]
	&\leq\dfrac{1}{2\epsilon n}\sum_{i\in\mathbb{Z}\cap n I_x^{\epsilon}}(-J(\tau_{i}\omega,\delta(\tau_{i}\omega)))\chi_{(-J(\tau_{i}\omega,\delta(\tau_{i}\omega))\leq k)}%\\[1mm]
	%&
	\leq k\cdot\dfrac{1}{2\epsilon n}\sum_{i\in\mathbb{Z}\cap n I_x^{\epsilon}}\chi_{(-J(\tau_{i}\omega,\delta(\tau_{i}\omega))\leq k)}.
	\end{split}
	\end{align}
	From Proposition~\ref{Prop:ergodicbeta}, we get for $k\in\mathbb{Q}$ and all $x\in\mathbb{R}$
	\begin{align*}
	\dfrac{1}{2\epsilon n}\sum_{i\in\mathbb{Z}\cap n I_x^{\epsilon}}\chi_{(-J(\tau_{i}\omega,\delta(\tau_{i}\omega))\leq k)}\rightarrow\mathbb{E}\left[\chi_{(-J(\delta)\leq k)} \right]\quad\text{as}\ n\rightarrow\infty,
	\end{align*}
	independent of $x$ and $\omega$, where $J(\delta)$ represents the random variable $\omega\mapsto J(\omega,\delta(\omega))$. Since we consider a convergent subsequence of $\beta_n(\omega,x,\epsilon)$, the passage to  $n\rightarrow\infty$ in \eqref{betaherleitung1} yields for $k\in\mathbb{Q}$
	\begin{align}
	\label{betaherleitung2}
	\lim\limits_{n\rightarrow\infty}\inf_{0\leq i\leq n-1}\left\{-J(\tau_{i}\omega,\delta(\tau_{i}\omega))\, :\, |i\lambda_n-x|< \epsilon \right\}\cdot\mathbb{E}\left[\chi_{(-J(\delta)\leq k)} \right]\leq k\cdot \mathbb{E}\left[\chi_{(-J(\delta)\leq k)} \right].	
	\end{align}
	For $k>\inf\left\{-J(\omega,\delta(\omega))\, :\, \omega\in\Omega \right\}=\beta$, it holds true that
	%\begin{align*}
	$\mathbb{E}\left[\chi_{(-J(\delta)\leq k)} \right]=\mathbb{P}\left(\left\{-J(\delta)\leq k\right\} \right)>0$.	
	%\end{align*}
	Therefore, we divide by the expectation value in \eqref{betaherleitung2} and obtain for $k>\beta$, $k\in\mathbb{Q}$
	\begin{align*}
	\lim\limits_{n\rightarrow\infty}\beta_n(\omega,x,\epsilon)=\lim\limits_{n\rightarrow\infty}\inf_{0\leq i\leq n-1}\left\{-J(\tau_{i}\omega,\delta(\tau_{i}\omega))\, :\, |i\lambda_n-x|< \epsilon \right\}\leq k.
	\end{align*}
	Further, we get for $k\in\mathbb{Q}$
	\begin{align*}
	\lim\limits_{n\rightarrow\infty}\beta_n(\omega,x,\epsilon)&=\lim\limits_{n\rightarrow\infty}\inf_{0\leq i\leq n-1}\left\{-J(\tau_{i}\omega,\delta(\tau_{i}\omega))\, :\, |i\lambda_n-x|< \epsilon \right\}\\[2mm]
	&=\liminf\limits_{k\searrow\beta }\lim\limits_{n\rightarrow\infty}\inf_{0\leq i\leq n-1}\left\{-J(\tau_{i}\omega,\delta(\tau_{i}\omega))\, :\, |i\lambda_n-x|< \epsilon \right\}\leq\lim\limits_{k\searrow\beta} k=\beta,
	\end{align*}
	which finishes the proof of \eqref{limesbeta} and therefore the proof of \eqref{limsup2}.\\
	
	Altogether, we have shown (b), namely
	\begin{align}
	\label{attouchskaliert1}
	\limsup\limits_{n\rightarrow\infty}E_n^{\gamma_n}(\omega,v_n)\leq\underline{\alpha}\int_{0}^{1}|v'(x)|^2\,\mathrm{d}x+\beta=E^{\gamma}(v).
	\end{align}
	
	\noindent
	(c) It is left to show that $v_n\rightarrow v$ in $L^1(0,1)$. For this, we split the integral as
	\begin{align}
	\label{integrale}
	\begin{split}
	&\lVert v_n-v\rVert_{L^1(0,1)} %=\int_{0}^{1}|v_n(x)-v(x)|\,\mathrm{d}x%\\[1mm]
	%&
	=\int_{0}^{h_n^{\epsilon}\lambda_n}|v_n(x)-v(x)|\,\mathrm{d}x+\int_{h_n^{\epsilon}\lambda_n}^{T_n\lambda_n}|v_n(x)-v(x)|\,\mathrm{d}x+\int_{T_n\lambda_n}^{1}|v_n(x)-v(x)|\,\mathrm{d}x
	\end{split}
	\end{align}
	and consider each integral separately in parts (i) to (iii) below. Later, in part (iv), we combine the results from (i) to (iii) with the Attouch Lemma \cite[Corollary~1.16]{Attouch1984}. 

\medskip 	
	\noindent
	(i) For the first integral in \eqref{integrale}, we obtain
	\begin{align*}
	\int_{0}^{h_n^{\epsilon}\lambda_n}|v_n(x)-v(x)|\,\mathrm{d}x&=\int_{0}^{h_n^{\epsilon}\lambda_n}|v(0^-)-v(0^+)|\,\mathrm{d}x%\\[1mm]
	%&
	=|v(0^-)-v(0^+)|h_n^{\epsilon}\lambda_n\leq |v(0^-)-v(0^+)|\epsilon.
	\end{align*}	
	
	\noindent
	(ii) For the second integral in \eqref{integrale}, we get
	\begin{align*}
	&\int_{h_n^{\epsilon}\lambda_n}^{T_n\lambda_n}|v_n(x)-v(x)|\,\mathrm{d}x\\[2mm]
	&=\int_{(h_n^{\epsilon}+1)\lambda_n}^{T_n\lambda_n}\left|v(0^+)+\gamma_n-\gamma+z\left(T_n\lambda_n-\rho \right)-v(0^+)\right|\,\mathrm{d}x\\[2mm]
	&\hspace{10mm}+\int_{h_n^{\epsilon}\lambda_n}^{(h_n^{\epsilon}+1)\lambda_n}\left|\dfrac{v(0^+)+\gamma_n-\gamma+z\left(T_n\lambda_n-\rho \right)}{\lambda_n}\left(x-h_n^{\epsilon}\lambda_n \right)-v(0^+) \right|\,\mathrm{d}x\\[2mm]
	&=\left|\gamma_n-\gamma+z\left(T_n\lambda_n-\rho \right)\right|\left(T_n-h_n^{\epsilon}-1\right)\lambda_n\\[1mm]
	&\hspace{10mm}+\left|v(0^+)+\gamma_n-\gamma+z\left(T_n\lambda_n-\rho \right) \right|\left(2\lambda_nh_n^{\epsilon}+\frac{1}{2}\lambda_n\right)-\lambda_nv(0^+) \\[1mm]
%	&\leq \left|\gamma_n-\gamma+z\left(T_n\lambda_n-\rho \right)\right|\left(T_n-h_n^{\epsilon}-1\right)\lambda_n\\[1mm]
%	&\hspace{10mm}+\left|v(0^+)+\gamma_n-\gamma +z\left(T_n\lambda_n-\rho \right)\right|\left(2\epsilon+\frac{1}{2}\lambda_n\right)-\lambda_nv(0^+)\\[1mm]
	&\rightarrow 2v(0^+)\epsilon\quad\text{as}\quad n\rightarrow\infty,
	\end{align*}
	since $h_n^{\epsilon}\lambda_n$ is bounded by $\epsilon$, $\gamma_n\rightarrow\gamma$ and $T_n\lambda_n\rightarrow\rho$.
	
	\medskip 
	
	\noindent
	(iii) The last integral in \eqref{integrale},
	%\begin{align*}
	$\int_{T_n\lambda_n}^{1}|v_n(x)-v(x)|\,\mathrm{d}x$,
%	\end{align*}
	is the most interesting one. With $\epsilon_{j,0}=1$ for $j=0$ and $\epsilon_{j,0}=0$ for $j>0$, we get
	\begin{align}
	\label{ieta1}
	\begin{split}
	&\int_{T_n\lambda_n}^{1}|v_n(x)-v(x)|\,\mathrm{d}x=\sum_{j=0}^{j_{\mathrm{max}}}\int_{(i_{\mathrm{min}}^{j,n}-1)\lambda_n}^{i_{\mathrm{max}}^{j,n}\lambda_n}|v_n(x)-v(x)|\,\mathrm{d}x\\[2mm]
	&=\sum_{j=0}^{j_{\mathrm{max}}}\int_{(i_{\mathrm{min}}^{j,n}-1)\lambda_n}^{i_{\mathrm{max}}^{j,n}\lambda_n}\left|\gamma_n-\gamma+z\left(T_n\lambda_n-\rho\right)\epsilon_{j,0}+ \int_{(i_{\mathrm{min}}^{j,n}-1)\lambda_n}^{x}v_n'(y)-v'(y)\,\mathrm{d}y\right|\,\mathrm{d}x\\[2mm]
	&\leq\sum_{j=0}^{j_{\mathrm{max}}}\int_{(i_{\mathrm{min}}^{j,n}-1)\lambda_n}^{i_{\mathrm{max}}^{j,n}\lambda_n}\int_{(i_{\mathrm{min}}^{j,n}-1)\lambda_n}^{i_{\mathrm{max}}^{j,n}\lambda_n}|v_n'(y)-v'(y)|\,\mathrm{d}y\,\mathrm{d}x+\sum_{j=0}^{j_{\mathrm{max}}}\int_{(i_{\mathrm{min}}^{j,n}-1)\lambda_n}^{i_{\mathrm{max}}^{j,n}\lambda_n}\left|\gamma_n-\gamma+z\left(T_n\lambda_n-\rho\right)\epsilon_{j,0}\right|\,\mathrm{d}x\\[2mm]
	&\leq\sum_{j=0}^{j_{\mathrm{max}}}\left(i_{\mathrm{max}}^{j,n}-i_{\mathrm{min}}^{j,n}+1\right)\lambda_n\int_{(i_{\mathrm{min}}^{j,n}-1)\lambda_n}^{i_{\mathrm{max}}^{j,n}\lambda_n}|v_n'(x)-v'(x)|\,\mathrm{d}x+\dfrac{1-\rho}{\mu}\left(\mu+\lambda_n\right)\left(|\gamma_n-\gamma|+z\lambda_n \right)\\[2mm]
	&\leq \sum_{j=0}^{j_{\mathrm{max}}}(\mu+\lambda_n)\int_{(i_{\mathrm{min}}^{j,n}-1)\lambda_n}^{i_{\mathrm{max}}^{j,n}\lambda_n}|v_n'(x)-v'(x)|\,\mathrm{d}x+\dfrac{1-\rho}{\mu}\left(|\gamma_n-\gamma|+z\lambda_n \right),
	\end{split}
	\end{align}
	because we have $\rho-T_n\lambda_n\leq\lambda_n$ and $\lambda_n\left(i_{\mathrm{max}}^{j,n}-i_{\mathrm{min}}^{j,n}+1 \right)\leq\mu+\lambda_n$. The integral in the last row of \eqref{ieta1} is considered separately. For $j>0$, we have $v'(x)=z$ and therefore get	
	\begin{align}
	\label{ieta2}
	\begin{split}
	&\int_{(i_{\mathrm{min}}^{j,n}-1)\lambda_n}^{i_{\mathrm{max}}^{j,n}\lambda_n}|v_n'(x)-v'(x)|\,\mathrm{d}x\\[1mm]
	&=\sum_{i\in I_{j,n}^{\mu*}}\int_{i\lambda_n}^{(i+1)\lambda_n}|v_n'(x)-v'(x)|\,\mathrm{d}x=\sum_{i\in I_{j,n}^{\mu*}}\lambda_n\left|\left(\dfrac{1}{\left|I_{j,n}^{\mu*} \right|}\sum_{k\in I_{j,n}^{\mu*}}\dfrac{1}{\alpha(\tau_k\omega)}\right)^{-1}\hspace{-4mm}z\dfrac{1}{\alpha(\tau_i\omega)}-z \right|\\[2mm]
	&\leq\lambda_n|z|\sum_{i\in I_{j,n}^{\mu*}}\left(\left(\dfrac{1}{\left|I_{j,n}^{\mu*} \right|}\sum_{k\in I_{j,n}^{\mu*}}\dfrac{1}{\alpha(\tau_k\omega)} \right)^{-1}\hspace{-4mm}\dfrac{1}{\alpha(\tau_i\omega)}+1\right)\\[2mm]
	&\leq\lambda_n|z|\left(\left(\dfrac{1}{\left|I_{j,n}^{\mu*} \right|}\sum_{k\in I_{j,n}^{\mu*}}\dfrac{1}{\alpha(\tau_k\omega)} \right)^{-1}\cdot\left(\sum_{i\in I_{j,n}^{\mu*}}\dfrac{1}{\alpha(\tau_i\omega)}\right)+\sum_{i\in I_{j,n}^{\mu*}}1\right)\\[1mm]
	&=2|z|\lambda_n\left|I_{j,n}^{\mu*} \right| \leq C\cdot\left(\mu+\lambda_n\right),
	\end{split}
	\end{align}
	since it holds true that $\lambda_n\left|I_{j,n}^{\mu*} \right| =\lambda_n\left(i_{\mathrm{max}}^{j,n}-i_{\mathrm{min}}^{j,n}+1 \right)\leq\mu+\lambda_n$. Now that we have determined the integral in the last row of \eqref{ieta1} for $j>0$, we calculate it for $j=0$ by
	\begin{align}
	\label{ietazwischen}
	\begin{split}
	&\int_{T_n\lambda_n}^{i_{\mathrm{max}}^{0,n}\lambda_n}|v_n'(x)-v'(x)|\,\mathrm{d}x\\[1mm]
	&=\int_{T_n\lambda_n}^{\rho}|v_n'(x)-v'(x)|\,\mathrm{d}x+\int_{\rho}^{(T_n+1)\lambda_n}|v_n'(x)-v'(x)|\,\mathrm{d}x+\int_{(T_n+1)\lambda_n}^{i_{\mathrm{max}}^{0,n}\lambda_n}|v_n'(x)-v'(x)|\,\mathrm{d}x.
	\end{split}
	\end{align}
	The first addend of \eqref{ietazwischen} is
	\begin{align*}
	\int_{T_n\lambda_n}^{\rho}|v_n'(x)-v'(x)|\,\mathrm{d}x&=\int_{T_n\lambda_n}^{\rho}\left|\left(\dfrac{1}{\left|I_{j,n}^{\mu*} \right|}\sum_{k\in I_{j,n}^{\mu*}}\dfrac{1}{\alpha(\tau_k\omega)}\right)^{-1}\hspace{-4mm}z\dfrac{1}{\alpha\left(\tau_{T_n}\omega\right)}-0 \right|\,\mathrm{d}x\\[1mm]
	&=\left(\rho-T_n\lambda_n\right)\left(\dfrac{1}{\left|I_{j,n}^{\mu*} \right|}\sum_{k\in I_{j,n}^{\mu*}}\dfrac{1}{\alpha(\tau_k\omega)}\right)^{-1}\hspace{-4mm}|z|\dfrac{1}{\alpha\left(\tau_{T_n}\omega\right)},
	\end{align*}
	which converges to zero as $n\to\infty$ because of the convergence $T_n\lambda_n\to\rho$, Proposition~\ref{Prop:averages2} and the boundedness of $\alpha^{-1}(\omega)$ due to Remark~\ref{Rm:integrabilityandexpectationvalues2}. The second addend of \eqref{ietazwischen} is
	\begin{align*}
	\int_{\rho}^{(T_n+1)\lambda_n}|v_n'(x)-v'(x)|\,\mathrm{d}x&=\int_{\rho}^{(T_n+1)\lambda_n}\left|\left(\dfrac{1}{\left|I_{j,n}^{\mu*} \right|}\sum_{k\in I_{j,n}^{\mu*}}\dfrac{1}{\alpha(\tau_k\omega)}\right)^{-1}\hspace{-4mm}z\dfrac{1}{\alpha\left(\tau_{T_n}\omega\right)}-z \right|\,\mathrm{d}x\\[2mm]
	&=\left((T_n+1)\lambda_n-\rho\right)|z|\left|\left(\dfrac{1}{\left|I_{j,n}^{\mu*} \right|}\sum_{k\in I_{j,n}^{\mu*}}\dfrac{1}{\alpha(\tau_k\omega)}\right)^{-1}\hspace{-4mm}\dfrac{1}{\alpha\left(\tau_{T_n}\omega\right)}-1\right|,
	\end{align*}
	which again converges to zero as $n\to\infty$ because of the convergence $(T_n+1)\lambda_n\to\rho$, Proposition~\ref{Prop:averages2} and the boundedness of $\alpha^{-1}(\omega)$ due to Remark~\ref{Rm:integrabilityandexpectationvalues2}. For the last addend of \eqref{ietazwischen}, we reuse the calculations from \eqref{ieta2} and get
	\begin{align*}
	&\int_{(T_n+1)\lambda_n}^{i_{\mathrm{max}}^{0,n}\lambda_n}|v_n'(x)-v'(x)|\,\mathrm{d}x=\sum_{i\in I_{j,n}^{\mu*}\setminus\{T_n\}}\int_{i\lambda_n}^{(i+1)\lambda_n}|v_n'(x)-v'(x)|\,\mathrm{d}x\\[1mm]
	&=\sum_{i\in I_{j,n}^{\mu*}}\lambda_n\left|\left(\dfrac{1}{\left|I_{j,n}^{\mu*} \right|}\sum_{k\in I_{j,n}^{\mu*}}\dfrac{1}{\alpha(\tau_k\omega)}\right)^{-1}\hspace{-4mm}z\dfrac{1}{\alpha(\tau_i\omega)}-z \right|-\lambda_n\left|\left(\dfrac{1}{\left|I_{j,n}^{\mu*} \right|}\sum_{k\in I_{j,n}^{\mu*}}\dfrac{1}{\alpha(\tau_k\omega)}\right)^{-1}\hspace{-4mm}z\dfrac{1}{\alpha\left(\tau_{T_n}\omega\right)}-z \right|\\[2mm]
	&\leq2|z|\lambda_n\left|I_{j,n}^{\mu*} \right|+\hat{C}\lambda_n \leq C\left(\mu+\lambda_n\right),
	\end{align*}
	where the bound $\hat{C}$ is due to the boundedness of $\alpha^{-1}(\omega)$ by Remark~\ref{Rm:integrabilityandexpectationvalues2} and the convergence of the sum according to Proposition~\ref{Prop:averages2}. Altogether, this yields, for \eqref{ietazwischen} and for $n$ large enough,
	\begin{align}
	\label{ietaj0}
	\int_{T_n\lambda_n}^{i_{\mathrm{max}}^{0,n}\lambda_n}|v_n'(x)-v'(x)|\,\mathrm{d}x\leq C\left(\mu+\lambda_n \right).
	\end{align}
	Combining \eqref{ieta1}, \eqref{ieta2} and \eqref{ietaj0}, we obtain
	\begin{align*}
	%\label{limsupconvergenz}
	\begin{split}
	&\int_{T_n\lambda_n}^{1}|v_n(x)-v(x)|\,\mathrm{d}x\leq\sum_{j=0}^{j_{\mathrm{max}}}(\mu+\lambda_n) C(\mu+\lambda_n)+\dfrac{1-\rho}{\mu}\left(|\gamma_n-\gamma|+z\lambda_n \right)\\[1mm]
	&=\frac{1-\rho}{\mu}(\mu+\lambda_n) C (\mu+\lambda_n)+\dfrac{1-\rho}{\mu}\left(|\gamma_n-\gamma|+z\lambda_n \right)\\[1mm]
	&\leq\tilde{C}\left(\mu+2\lambda_n+\frac{\lambda_n^2}{\mu} \right)+\dfrac{1-\rho}{\mu}\left(|\gamma_n-\gamma|+z\lambda_n \right)\to \tilde{C}\mu \quad\text{as}\ n\to\infty.
	\end{split}
	\end{align*} 	
	
	\noindent
	(iv) Altogether, we have shown in the parts (i)--(iii) that
	\begin{align}
	\label{attouchskaliert2}
	\limsup\limits_{n\to\infty}\lVert v_n-v\rVert_{L^1(0,1)}\leq \hat{C}\epsilon+\tilde{C}\mu.
	\end{align}
	Now, by setting  $\epsilon=\mu$, we combine the results from \eqref{attouchskaliert1} and \eqref{attouchskaliert2} and get (recall that $v_n$ strictly accurately is $v_{n,\mu}$)
	\begin{align*}
	\limsup\limits_{\mu\rightarrow0}\limsup\limits_{n\rightarrow\infty}\bigg(|E_n^{\gamma_n}(\omega,v_{n,\mu})-E^{\gamma}(v)|+\lVert v_{n,\mu}-v\lVert_{L^1(0,1)} \bigg)=0.
	\end{align*}
	From the Attouch Lemma \cite[Corollary~1.16]{Attouch1984}, we therefore get the existence of a subsequence $\mu_n$ with $\mu_n\to0$ as $n\to\infty$ and
	\begin{align*}
	0&\leq\limsup_{n\to\infty}\left(|E_n^{\gamma_n}(\omega,v_{n,\mu_n})-E^{\gamma}(v)|+\lVert v_{n,\mu_n}-v\lVert_{L^1(0,1)} \right)\\[1mm]
	&\leq \limsup\limits_{\mu\rightarrow0}\limsup\limits_{n\rightarrow\infty}\left(|E_n^{\gamma_n}(\omega,v_{n,\mu})-E^{\gamma}(v)|+\lVert v_{n,\mu}-v\lVert_{L^1(0,1)} \right)=0.
	\end{align*}
	Finally, this proves $\lVert v_{n,\mu_n}-v\rVert_{L^1(0,1)}\to 0$ as $n\to\infty$, which concludes (c). Hence, $(v_{n,\mu_n})$ is the recovery sequence for the affine function $v(x)=zx$, which was the goal of step A.
	
	This construction of a recovery sequence for affine functions with a jump can easily be extended to piecewise affine functions with jumps by dividing the interval $[0,1]$ into parts where the function is affine.
	
	\medskip 
	
	\noindent
	\textit{Step B: Smooth functions, constant near the jump.}
	
	We have constructed a recovery sequence for piecewise affine functions with jumps. With this result, we get a recovery sequence for every $v\in C^2([0,1]\setminus S_{v})$ where $v$ is constant on $x\in[x_0-\eta,x_0+\eta]$ with $S_v=\{x_0\}$ and $\eta>0$ small enough. This is justified as follows: on $x\in[x_0-\eta,x_0+\eta]$, $v$ is affine. On $[0,x_0-\eta]$ and $[x_0+\eta,1]$, we take, for $\delta>0$, the piecewise affine interpolation $v_N$ of $v$ with grid points $(t_j^{N})_{j=0,\ldots,a_N,b_N,\ldots,j_N}$ with $t_0=0$, $t_{a_N}=x_0-\eta$, $t_{b_N}=x_0+\eta$, $t_{j_N}=1$ and $\delta<t_{j+1}^N-t_j^{N}<2\delta$ for $j=0,\ldots,a_{N}-1,b_N,\ldots,j_N$. Note that for $\delta\to0$ we also get $N\to\infty$, which is the reason why we use both equivalently. Then, we get by the Jensen inequality
	\begin{align}
	\label{edeltavn}
	\begin{split}
	E^{\gamma}(v)&=\underline{\alpha}\int_{0}^{1}|v'(x)|^2\,\mathrm{d}x+\beta %\\[1mm]
	%&
	=\underline{\alpha}\sum_{j=0}^{N}\left(t_{j}^N-t_{j-1}^N\right)\dfrac{1}{t_{j}^N-t_{j-1}^N}\int_{t_{j-1}^N}^{t_j^N}|v'(x)|^2\,\mathrm{d}x+\beta\\[1mm]
	&\geq\underline{\alpha}\sum_{j=0}^{N}\left(t_{j}^N-t_{j-1}^N\right)\left|\dfrac{1}{t_{j}^N-t_{j-1}^N}\int_{t_{j-1}^N}^{t_j^N}v'(x)\,\mathrm{d}x\right|^2+\beta \\[2mm]
	&=\underline{\alpha}\sum_{j=0}^{N}\left(t_{j}^N-t_{j-1}^N\right)\left|\dfrac{v(t_j^N)-v(t_{j-1}^N)}{t_{j}^N-t_{j-1}^N}\right|^2+\beta\\[1mm]
	&=\underline{\alpha}\sum_{j=0}^{N}\int_{t_{j-1}^N}^{t_j^N}|v_N'(x)|^2\,\mathrm{d}x+\beta=\underline{\alpha}\int_{0}^{1}|v_N'(x)|^2\,\mathrm{d}x+\beta=E^{\gamma}(v_N).
	\end{split}
	\end{align}
	We argue as in step E of the proof of the liminf inequality in \cite[Theorem~3.1]{unserPaper1} to get $v_N\rightarrow v$ in $L^1(0,1)$. Further, the $\Gamma$-$\limsup$ is lower semicontinuous. Therefore, we get
	\begin{align*}
	\Gamma\text{-}\limsup_{n\rightarrow\infty}E_n^{\gamma_n}(\omega,v)\stackrel{l.s.c.}{\leq}\liminf_{N\rightarrow \infty}\left(\Gamma\text{-}\limsup_{n\rightarrow\infty}E_n^{\gamma_n}(\omega,v_N) \right)\stackrel{(*)}{\leq}\liminf_{N\rightarrow \infty}E^{\gamma}(v_N)\stackrel{\eqref{edeltavn}}{\leq}E^{\gamma}(v),
	\end{align*}
	where $(*)$ follows from the construction of the recovery sequence for piecewise affine functions in step A, which keeps the boundary values, see~\cite[Section 2.4, Corollary 3]{Burenkov}.\\
	
	\noindent
	\textit{Step C: Smooth functions.}
	
	Now that we have a recovery sequence for $v\in C^2([0,1]\setminus S_{v})$ where $v$ is constant on $x\in[x_0-\eta,x_0+\eta]$, we can extend it to functions $v=v_c+v_j$ with $v_c\in C^2[0,1]$ and $v_j$ is piecewise constant, which concludes the limsup-inequality. Without loss of generality, we set $S_v=\{x_0\}$. Now, we define, for $\eta>0$ small enough, an approximation $v_c^{\eta}$ with
	\begin{align*}
	v_c^{\eta}(x):=\begin{cases}
	v_c(x)&\quad\text{for}\quad x<x_0-\eta,\\[1mm]
	v_c(x_0-\eta)&\quad\text{for}\quad x\in [x_0-\eta,x_0+\eta],\\[1mm]
	v_c(x)-v_c(x_0+\eta)+v_c(x_0-\eta)&\quad\text{for}\quad x>x_0+\eta.
	\end{cases}
	\end{align*}
	Then, $v^{\eta}=v_c^{\eta}+v_j$ has two properties, namely (a) $v^{\eta}\rightarrow v$ in $L^1(0,1)$ for $\eta\rightarrow0$ and (b) $\int_{0}^{1}|{v_c^{\eta}}'(x)|^2\,\mathrm{d}x\rightarrow\int_{0}^{1}|v_c'(x)|^2\,\mathrm{d}x$ for $\eta\rightarrow0$, which can be seen as follows.
	
	\medskip
	
	\noindent
	(a) %We prove $v^{\eta}\rightarrow v$ in $L^1(0,1)$ for $\eta\rightarrow0$:
	Recalling that $v_c\in C^2[0,1]$, we deduce
	\begin{align*}
	&\int_{0}^{1}|v^{\eta}(x)-v(x)|\,\mathrm{d}x=\int_{x_0-\eta}^{x_0+\eta}|v_c(x_0-\eta)-v_c(x)|\,\mathrm{d}x+\int_{x_0+\eta}^{1}|v_c(x_0-\eta)-v_c(x_0+\eta)|\,\mathrm{d}x\\[2mm]
	&\leq 2\eta|v_c(x_0-\eta)|+\int_{x_0-\eta}^{x_0+\eta}|v_c(x)|\,\mathrm{d}x+(1-x_0-\eta)|v_c(x_0-\eta)-v_c(x_0+\eta)|\\[2mm]
	&\rightarrow 2\cdot0\cdot |v_c(x_0)|+0+(1-x_0)\cdot|v_c(x_0)-v_c(x_0)|=0\quad\text{for}\ \eta\rightarrow0.
	\end{align*}
	
	\noindent
	(b) We observe that %We prove $\int_{0}^{1}|{v_c^{\eta}}'(x)|^2\,\mathrm{d}x\rightarrow\int_{0}^{1}|v_c'(x)|^2\,\mathrm{d}x$ for $\eta\rightarrow0$:
	\begin{align*}
	\int_{0}^{1}|{v_c^{\eta}}'(x)|^2\,\mathrm{d}x&=\int_{0}^{x_0-\eta}|v_c'(x)|^2\,\mathrm{d}x+\int_{x_0+\eta}^{1}|v_c'(x)|^2\,\mathrm{d}x\rightarrow\int_{0}^{1}|v_c'(x)|^2\,\mathrm{d}x\quad\text{for}\ \eta\rightarrow0.
	\end{align*}

	Similarly to standard density arguments, we get with the properties (a) and (b)
	\begin{align*}
	\Gamma\text{-}\limsup_{n\rightarrow\infty}E_n^{\gamma_n}(\omega,v)\stackrel{\text{(a)+l.s.c.}}{\leq}\liminf_{\eta\rightarrow 0}\left(\Gamma\text{-}\limsup_{n\rightarrow\infty}E_n^{\gamma_n}(\omega,v^{\eta}) \right)\stackrel{(*)}{\leq}\liminf_{\eta\rightarrow 0}E^{\gamma}(v^{\eta})\stackrel{\text{(b)}}{\leq}E^{\gamma}(v),
	\end{align*}
	where $(*)$ follows from the construction of the recovery sequence from step B, which also keeps the boundary values. As was already discussed in the beginning of the limsup-part, the construction of the recovery sequence is now completed also for an arbitrary $v\in SBV_c^{\gamma}$.
	
	\medskip

	\noindent
	\textbf{Step 3.} Convergence of minimum problems.
	
	The convergence of minimum problems follows from the coerciveness of $E_n^{\gamma_n}$ and the $\Gamma$-convergence result due to the main theorem of $\Gamma$-convergence. It is left to show that
	\begin{align}
	\label{minimumproblemskaliert}
	\min_v E^{\gamma}(v)=\min\{\underline{\alpha}\gamma^2,\beta\}.
	\end{align}
	This is done analogously to \cite{ScardiaSchloemerkemperZanini2012}. For $\gamma>0$ fixed and $v$ with boundary conditions $v(0)=0$ and $v(1)=\gamma$ and fulfilling $[v]>0$ on $S_v$ we have to distinguish two cases: First, let $S_v=\emptyset$, then we have $v\in W^{1,1}(0,1)$ and with the Jensen-inequality we get
	\begin{align*}
	\underline{\alpha}\int_{0}^{1}|v'(x)|^2\,\mathrm{d}x \geq \underline{\alpha}\left|\int_{0}^{1}v'(x)\right|^2=\underline{\alpha}\gamma^2	
	\end{align*}
	and therefore with the minimizer $v(x)=\gamma x$
	\begin{align*}
	\min_v E^{\gamma}(v)=\min_v\left\{\underline{\alpha}\int_{0}^{1}|v'(x)|^2\,\mathrm{d}x \right\}=\underline{\alpha}\gamma^2.
	\end{align*}
	Second, for $S_v\neq\emptyset$, we get because of $\underline{\alpha}>0$
	\begin{align*}
	\min_v E^{\gamma}(v)=\min_v\left\{\underline{\alpha}\int_{0}^{1}|v'(x)|^2\,\mathrm{d}x+\beta\#S_v \right\}=\beta
	\end{align*}
	where the minimizer has one jump point $S_v=\{t\}$ for some $t\in [0,1]$ and is given by
	\begin{align*}
	v(x)=\begin{cases}
	0&\quad\text{if}\ x\in[0,t),\\
	\gamma&\quad\text{if}\ x\in[t,1].
	\end{cases}
	\end{align*}
	This shows \eqref{minimumproblemskaliert} and completes the proof of Theorem~\ref{Thm:rescaled}.
\end{proof}

\textbf{Acknowledgments.} The largest part of this work was performed while LL was affiliated with the Institute of Mathematics, University of Würzburg. LL gratefully acknowledges the kind hospitality of the Technische Universität Dresden during her research visits, which were partial funded by the Deutsche Forschungsgemeinschaft
(DFG, German Research Foundation) – within project 405009441 and TU Dresden’s
Institutional Strategy “The Synergetic University”. AS would like to thank the Isaac Newton Institute for Mathematical Sciences for support and hospitality during the programme
``The Mathematical Design of New Materials''
when some work on this paper was undertaken. This programme was supported by EPSRC grant number EP/R014604/1.

\end{document}